%% file: CVM.tex
\documentclass[a4paper,11pt]{article}

\usepackage{colortbl}
\usepackage{multirow}%
\usepackage{fullpage}
\usepackage{amsmath}%
\usepackage{amsthm}%
\usepackage{amsfonts}%
\usepackage{amssymb}%
\usepackage{graphicx}
\usepackage{enumerate}
\usepackage{dsfont}
\usepackage[affil-it]{authblk}

\newtheorem{theorem}{Theorem}
\newtheorem{corollary}{Corollary}
\newtheorem{proposition}[theorem]{Proposition}
\newtheorem{lemma}[theorem]{Lemma}

\newtheorem{remark}{Remark}

\newcommand{\indep}{\perp\hspace{-.25cm}\perp}

\renewcommand{\P}{\mathbb{P}}
\newcommand{\R}{\mathbb{R}}
\newcommand{\w }[1]{\widehat{#1}}
\renewcommand{\r}{ \rightarrow }
\newcommand{\lr}{ \longrightarrow }

\DeclareMathOperator{\spann}{span}
\DeclareMathOperator{\var}{var}
\DeclareMathOperator{\cov}{cov}

\DeclareMathOperator{\vecv}{vec}

\renewcommand{\t}[1]{\text{#1}}

\begin{document}

\title{An empirical process view of inverse regression}

 \author{ Fran\c cois Portier\footnote{Institut de statistique, biostatistique et sciences actuarielles, Universit\'{e} catholique de Louvain, Voie du Roman
Pays 20, B1348 Louvain-la-Neuve, Belgium.
Research supported by Fonds de la Recherche Scientifique (FNRS) A4/5 FC 2779/2014-2017 No.\ 22342320.
Email addresses: \texttt{francois.portier@uclouvain.be}.
  }
 }



\maketitle

\textbf{Abstract:}\ Most of the methods among the inverse regression literature rely on a slicing of the range of
the response variable. Theoretical results are 
usually shown assuming that (i) the slices are fixed while in practice estimators are constructed with 
(ii) random slices that contain the same number of observations. In this paper we obtain the asymptotic 
normality in the case where the slices contains the same number of observations. This issue matter 
since we find a gap between the asymptotic distributions related to both approaches (i) and (ii). Along this 
line, we revisit the asymptotic properties of existing methods such as sliced inverse regression and cumulative inverse regression,
and we also introduce a bootstrap procedure that reproduce accurately the law of certain Cram\'er-von Mises test statistics.
Our approach is based on the stochastic analysis of some empirical processes that lie 
close to a certain subspace of interest called the central subspace.

\bigskip


\section{Introduction}

Dimension reduction is a powerful tool usually employed to synthesise the dependence between two sets of random variables, say $(X,Y)$ where $X\in \R^p$ is called the vector of predictors and $Y\in \R$ is the variable to explain also called the response variable. Dimension reduction can be used to visualize the dependence in high dimensional data \cite{cook1998}, as well as to construct accurate estimators of the conditional distribution of $Y$ knowing $X$ \cite{hardle1989}. The most common way to model dimension reduction is to assume a certain structure on the conditional distribution of $Y$ given $X$ (see for instance the introduction of \cite{cook2002}). Here we assume that there exists $\beta_0 \in \R^{p\times d_0} $ such that the joint distribution of $(X,Y)$ satisfies
\begin{align}\label{model}
P(Y\in A|X) = P(Y\in A|\beta_0^T X),
 \end{align}
for every Borel set $A\subset \R$. The objective is to estimate the matrix $\beta_0$ or rather, because of identifiability reasons \cite{li1991}, the subspace it generates. This subspace is called the central subspace. To this typical semi-parametric problem, many different approaches have been investigated in the past decades \cite{hardle1989}, \cite{juditsky2001}, \cite{li1991}, \cite{cook2005}. In this paper we follow the idea of inverse regression introduced by Li \cite{li1991}. In spite of suffering from theoretical restriction on $X$ inverse regression often leads to estimators that are very accurate and computationally efficient. Inverse regression methods are widely spread probably because they provide a reasonable trade-off between accuracy and complexity.

Throughout the paper, we will assume that the central subspace is unique. This is known to be true as soon as $X$ has a density (\cite{portier2013}, Theorem 1). For more clarity in the statements we introduce the standardized predictors $Z=\Sigma^{-1/2}(X-E X)$ with $\Sigma=\var(X)$. The standardized central subspace, generated by $\Sigma^{1/2} \beta_0$ is denoted by $E_c$ and we let $P$ be the orthogonal projector on $E_c$.

Inverse regression is based on the following assumption. We say that $X$ satisfies the linearity condition if
\begin{align}
E(Z|PZ) = PZ,\tag{LC}\label{lc}
\end{align}
examples of such distributions include Gaussian distributions, uniform distribution on the sphere, or more generally the class of spherical variables \cite{li1993}. Li noticed in \cite{li1991} that under (\ref{model}) and (\ref{lc}),
\begin{align}\label{caractcentral1}
E(Z|Y) \in  E_c,
\end{align}
with probability $1$. He then proposed to approximate $E_c$ by estimating the subspace generated by $\var(E(Z|Y))$. The estimation is realized through a slicing of the response $Y$. A similar slicing method that has been shown to be more efficient is the minimum discrepancy approach (MD) \cite{cook2005}. When facing regression models with a symmetric link function (often refereed as the SIR pathology), SIR is inconsistent. Li suggested in \cite{li1991} to use second order moments of the predictors. Following this idea, some authors have introduced order 2 moments methods as for instance sliced average variance estimation (SAVE) \cite{cook1991}, directional regression \cite{li2007} and order 2 optimal function \cite{portier2013}. These methods require an additional assumption called the constant covariance condition,
\begin{align}
\var(Z|PZ) = const., \tag{CCV}\label{ccv}
\end{align}
they are based on the result that, under (\ref{model}), (\ref{lc}) and  (\ref{ccv}), it holds that
\begin{align}\label{caractcentral2}
\var( Z|Y)-I \in E_c,
\end{align}
where $I$ is the identity matrix.

As a consequence of Equations (\ref{caractcentral1}) and (\ref{caractcentral2}), the current literature have put the focus on the estimation of subspaces that are generated by conditional quantities. A natural issue which arises is to know whether a nonparametric estimation  is really necessary. On the one hand, some authors have studied the limiting distribution of SIR and SAVE estimators as the slicing becomes more thin \cite{hsing1992}, \cite{zhu1995}, \cite{zhu1996}, \cite{zhu2007}. Though the conditional quantities $E(Z|Y) $ or $\var(Z|Y)$ can not be estimated at rates root $n$, these authors shown that the rate root $n$ is in fact available when estimating moments of these quantities, as for instance $\var(E(Z|Y))$. One the other hand, other authors considered a constant number of slices, so that the length of the slices does not go to $0$ \cite{cook1991}, \cite{cook2005}, \cite{li2007}, \cite{portier2013}. In favour of the latter approach, for order 1 moments methods, one might argue that since
\begin{align*}
E[Z\psi(Y)] \in E_c,
\end{align*}
for any measurable function $\psi$ such that $E[Z\psi(Y)]<\infty$, the whole space $E_c$ will eventually be recovered as soon as the number of function $\psi$ is large (see \cite{portier2013}, Theorem 3). Going further, a natural idea is to consider the estimation of $E_c$ when $\psi$ describe a given class of function without necessarily being a slicing. This can be found in \cite{yin2002}, where the use of polynomial functions are discussed, and in \cite{portier2013} where the optimal choice of $\psi$ among a Hilbert space is considered (see also \cite{gardes2009} for the use of basis functions). In \cite{zhu2010}, the authors consider a sum over a non-countable class of functions: the indicators of sets $\{Y\leq t\}$, when $t$ varies on the real line. The underlying method is an integral based method called cumulative slicing estimation (CUME). In this paper we continue along this line by providing an empirical process view of the problem, by indexing the estimators by the elements of a given class of function.

The first contribution of the paper is the introduction and the study of two empirical processes that get closer to $E_c$ as the number of observations increases, one is based on the first conditional moments of $Z$ knowing $Y$ and the other one rely on the second conditional moments of $Z$ knowing $Y$. Let $\Phi:\R\r [0,1] $ be a distribution function and denote by $\Phi^-$ its generalized inverse, given by
\begin{align*}
\Phi^-(u)=\inf \{t\in \R \ : \ \Phi(t)\geq u\},
\end{align*}
for every $u\in [0,1] $. We define the first moment process as
\begin{align*}
c_{\Phi} (u) = E(Z \mathds 1 _{\{Y\leq \Phi^{-}(u)\}}) ,
\end{align*}
and the second moment process as
\begin{align*}
C _{\Phi}(u) = E((ZZ^T-I) \mathds 1 {\{Y\leq  \Phi^{-}(u)\}}),
\end{align*}
for each $u\in  [0,1]$. Clearly, under (\ref{model}) and (\ref{lc}), $c_{\Phi}(u)\in  E_c$, if moreover (\ref{ccv}) holds then $C_{\Phi}(u)\in E_c$, for every $u\in [0,1]$.

The fact that $c_{\Phi}$ and $ C_{\Phi}$ are indexed by the class of indicator functions plays a key role in our analysis. First the class of indicators is large enough to ensure an exhaustive characterization of $E_c$.  Second it is sufficiently small to enjoy a small metric entropy which is at the root of many nice asymptotic properties of the associated empirical process \cite{vandervaart1996}. Such properties include weak convergence of estimators of $ c_{\Phi}$ and $ C_{\Phi}$ with root $n$ rates, and the validity of some general weighted bootstrap procedures.

The function $\Phi$ is a user-selected function. As in copula modelling, to alleviate the effect of the marginal distribution of $Y$ in the estimation, it is convenient to ``uniformize" the variable $Y$. This is done by choosing $\Phi$ equal to $F$: the cumulative distribution function (cdf) of $Y$. Since $F$ is unknown, such a choice involves a little more technicalities in the proof but it leads to an accurate and computationally simple rank-based estimator. In \cite{fermanian2004}, the authors studied the weak convergence of the empirical copula process. Following their approach, our theoretical study is based on both the delta-method for stochastic processes and a ``trick'' allowing us to consider $Y$ as uniformly distributed (see Remark \ref{remranktransform}).

The study of these processes is conducted in Section \ref{s2}. It shall be the basis of our study about inverse regression, the main point of which are outlined bellow.
\begin{enumerate}[i)]
\item (see Section \ref{s31} and \ref{s32}) We obtain the exact asymptotic distribution of SIR when the number of slices is fixed and each slice contains the same number of observations. This way of computing SIR was already pointed out in Remark 4.2 in \cite{li1991} and it is the most common way to compute slicing estimators. The main issue here is to account for the effect of the randomness of the slices on the asymptotic distribution of SIR. To our knowledge, such results are new in the literature.

\item (see Section \ref{s32}) We introduce the class of integral based methods that approximate $E_c$ through the range of the matrices
\begin{align*}
 \int \mu (u)\mu(u) ^T d\nu(u),
\end{align*}
where $\mu$ stands for a stochastic process that lies in $E_c$, e.g. $c_\Psi$ or $C_{\Psi}$, and $\nu$ is a given probability measure. We show that this class includes interesting members such as SIR and CUME. Under mild condition, we prove the asymptotic normality and we provide a valid bootstrap procedure that indeed accounts for the randomness of the slices. Bootstrap is made through a weighting of the estimators that follows from  \cite{wellner1993b}, it includes for instance Efron's orginal bootstrap or the Bayesian bootstrap. Even for SIR or CUME, no such bootstrap was available in the literature. 

\item (see Section \ref{s33}) In the same spirit as the integral based methods of (ii), we develop several statistical tests of the type Cram\'er-von Mises: (a) a test of the dimension of a model, i.e. $d_0=d$ against $  d_0>d$, for some $1\leq d \leq p$, (b) following \cite{cook2004}, a test to assess the no effect of some user-selected sets of predictors, 
say $\eta^TZ$ with $\eta\in \R^{p\times (p-d)}$, 
(c) a test similar to (b), but now with $\eta$ estimated by a given dimension reduction method. The latter might lead us to evaluate whether a model is subject to the SIR pathology. The limiting laws of the considered statistics are fairly hard to estimate so that we provide a valid Bootstrap procedure in order to compute their quantiles. The choice of the bootstrap is crucial in testing since the bootstrap statistic needs to behave similarly as the statistic under $H_0$ even if $H_1$ is realized \cite{hall1991}. To implement the bootstrap, we follow ideas from \cite{portier2014} where a constraint bootstrap was developed for testing the rank of a matrix.

\end{enumerate}

A numerical analysis is given in Section \ref{s4}, in which we study the behaviour of the bootstrap approximation in significance testing.

\section{Preliminary results on empirical processes}\label{s2}

\subsection{Definitions}\label{s21}

Using the outer integral, the author Hoffman-Jorgensen has defined a notion of weak convergence of random sequences valued in a metric space \cite{hoffman1991}. This allows some elements of the considered sequences to be non-measurable 
provided that their limits are. We equip the space $l^{\infty}(\R)$ of bounded real functions defined on $\R$ with the supremum norm $\|\cdot\|_{\infty}$. We consider in this paper weak convergence of random elements in $l^\infty(\R)$ in the sense of Hoffman-Jorgensen.
Let $(Z_i,Y_i)_{1\leq i\leq  n}$ be an i.i.d. sequence of random elements lying in $\R^2$ with law $P$. We say that a class of measurable functions $\mathcal F\subset l^{\infty}(\R)$ is $P$-Donsker if 
\begin{align*}
n^{-1/2} \sum_{i=1}^n (f(Z_i,Y_i) -E f(Z_1,Y_1)) \text{ converges weakly in }l^{\infty}(\R).
\end{align*}
A complete study of the notion of weak convergence in metric spaces and Donsker classes is proposed in \cite{vandervaart1996}. The following lemma will be useful in the next.

\begin{lemma}\label{lemproduct}
Assume that $EZ^2$ is finite, then $ \{(z,y) \mapsto z \mathds{1}_{(-\infty,t]} (y), \ t\in \R\}$ 
is $P$-Donsker.
\end{lemma}

\begin{proof}
Let $\mathcal G=\{(x,y) \mapsto x \mathds{1}_{(-\infty,t]} (y), \ t\in \R\}$. First, it is well-known that $\mathcal F = \{ \mathds 1_{(-\infty,t]} ,\ t\in \R\}$ is $P$-Donsker (see for instance \cite{vandervaart1996}, Example 2.5.4, page 129). In particular, the covering number of $\mathcal F$ is such that
\begin{align}\label{covbound}
N(\epsilon, \mathcal F,L_2(P)) \leq \frac{2}{\epsilon^2}.
\end{align}
 Second, since functions of $\mathcal G$ have the form $g=\phi(id,f)$ for some $f\in \mathcal F$, where $\phi(x,y)=xy$ and ${id}$ stands for the identity function, we can write
\begin{align*}
\mathcal G = \phi( \{id\},\mathcal F  ).
\end{align*}
Let $f_1$ and $f_2$ be functions in $\mathcal F$, since we have
\begin{align*}
|\phi\circ(id,f_1)(x,y) -\phi\circ(id,f_2)(x,y)|^2 =x^2(f_1(y) -f_2(y))^2,
\end{align*}
we can apply Theorem 2.10.20 page 199 in \cite{vandervaart1996} (the condition above corresponds to (2.10.19), an envelope for $\mathcal F$ is the function equal to $1$ everywhere). In view of the bound for the covering number of $\mathcal F$ given in (\ref{covbound}), and the fact that the covering number of a single element is $1$, the uniform entropy condition is checked, making the class $\mathcal G$ a $P$-Donsker class.

\end{proof}

\subsection[When $\Phi$ is known]{Asymptotic behaviour when $\Phi$ is known}\label{s22}

From now on, $(Z_i,Y_i)_{1\leq i\leq  n}$ is an i.i.d. sequence of random elements lying in $\R^{p}\times \R$ and drawn from model (\ref{model}) with $\var(Z_1)=I$ and $EZ_1=0$. We denote by $|\cdot |_2$ the Euclidean norm. In what follows, elements of interest belong to the space $l^{\infty}([0,1])^p$ that is (with a slight abuse of notation) the space of bounded $\R^p$-valued functions defined on $[0,1]$. The empirical processes that estimate the processes $c_{\Phi}$ and $C_{\Phi}$ are defined as follows, for every $u\in [0,1] $, by
\begin{align*}
\w c_{\Phi} (u) = \frac 1 n \sum_{i=1}^n Z_i \mathds 1 _{\{Y_i\leq \Phi^{-}(u)\}}\qquad \t{and}\qquad \w C_{\Phi} (u) = \frac 1 n \sum_{i=1}^n (Z_i  Z_i ^T-I) \mathds 1 _{\{Y_i \leq  \Phi^{-}(u)\}} .
\end{align*}
We introduce the matrix
\begin{align*}
\gamma_1(u,v)=  \cov\big(Z\mathds 1 _{\{Y\leq  \Phi^{-}(u)\}},Z\mathds 1 _{\{Y\leq  \Phi^{-}(v)\}} \big).
\end{align*}

\begin{theorem}\label{th1}
Assume that $E[|Z_1|_2^2]$ is finite and $\Phi$ is a cdf, then $ \sqrt n (\w c_{\Phi} -c_{\Phi})$ converges weakly in $l^{\infty}([0,1])^p$ to a tight Gaussian process with zero-mean and covariance function $\gamma_1$.
\end{theorem}

\begin{proof}
Each coordinate of the process $ \sqrt n (\w c_{\Phi} -c_{\Phi})$ can be written as $\sqrt n (\P_n-P) g$ where by Lemma \ref{lemproduct}, $g$ lies in a Donsker class. Because tightness is equivalent to tightness of each coordinates, it implies that the process $\sqrt n (\w c_{\Phi} -c_{\Phi}) $ is tight. The limiting process is then given by the limiting distribution of the finite dimensional laws obtained by the multivariate central limit theorem.
\end{proof}

We now obtain the weak convergence of $\sqrt n (\w C_{\Phi}-C_{\Phi})$. To state this we define the operator $\t{vec}$ that vectorizes a matrix by stacking its columns, and we introduce the matrix
\begin{align*}
\Gamma_1(u,v)=  \cov(\t{vec}(ZZ^T-I) \mathds 1 _{\{Y\leq  \Phi^-(u)\}},\t{vec}(ZZ^T-I) \mathds 1 _{\{Y\leq  \Phi^-(v)\}}).
\end{align*}

\addtocounter{corollary}{1}
\begin{corollary}\label{cor1}
Assume that $E[|Z_1|_2^4]$ is finite and $\Phi$ is a cdf, then $ \sqrt n (\w C_{\Phi} -C_{\Phi}) $ converges weakly in $l^{\infty}([0,1])^{(p\times p)}$ to a tight Gaussian process with zero-mean and covariance function $\Gamma_1$.
\end{corollary}
\begin{proof}
We apply Theorem \ref{th1} with $\text{vec}(ZZ^T-I)$ in place of $Z$.
\end{proof}

\subsection[When $\Phi$ is unknown]{Asymptotic behaviour when $\Phi$ is the cdf of $Y$}\label{s23}

We focus on the case where $\Phi=F$ the unknown distribution function of $Y$. Since
\begin{align*}
c_F(u) = E(Z\mathds 1_{\{Y\leq  F^-(u)\}})\qquad \t{and}\qquad C_F(u) = E((ZZ^T-I)\mathds 1_{\{Y\leq  F^-(u)\}}),
\end{align*}
this choice ``uniformizes'' the variable $Y$ and, as a consequence, vanishes the effect of the distribution of $Y$ on the estimation. Clearly, we can not follow the same path as previously since the estimation of $F$ will certainly affect the limiting process.  We introduce the empirical cdf
\begin{align*}
\w F(t) = n^{-1} \sum_{i=1}^n \mathds 1 _{\{Y_i\leq  t\}},
\end{align*}
defined for each $t\in \R$. Our estimators are plugged-in estimators, i.e. $c_F$ and $C_F$ are respectively estimated by $\w c_{\w F_{\t{}}}$ and $\w C_{\w F_{\t{}}}$ given by
\begin{align*}
\w c_{\w F_{\t{}}} (u) = \frac 1 n \sum_{i=1}^n Z_i \mathds 1 _{\{Y_i\leq \w F_{\t{}} ^-(u)\}}\qquad \t{and}\qquad \w C_{\w F_{\t{}}} (u) = \frac 1 n \sum_{i=1}^n (Z_i  Z_i ^T-I) \mathds 1 _{\{Y_i\leq  \w F_{\t{}} ^-(u)\}}.
\end{align*}

\begin{remark}\label{remranktransform}\normalfont
 An important point is that when $F$ is continuous, without loss of generality, the variables $Y_i$'s can be assumed to be uniformly distributed on $[0,1]$. Equivalently the limiting function $F$ can be assumed to be the identity function on $[0,1]$. To show this, first note that because $\w F$ is a c\`ad-l\`ag function that has $1/n$-jumps at each $Y_i$, it is easy to show that for any $u\in [0,1]$,
\begin{align}\label{formulasetequivalence}
\{Y_i\leq  \w F^-(u) \} \Leftrightarrow \{\w F (Y_i)< u+n^{-1} \}.
\end{align} 
Then we have that 
 \begin{align*}
 \w c_{{\w F}_{\t{}}} (u) = \frac 1 n \sum_{i=1}^n Z_i \mathds 1 _{\{\w F(Y_i) < u +n^{-1}\}}
 \end{align*}
and a similar expression holds for $ \w C_{{\w F}}$. This makes the previous estimators being sums over the $Z_i$'s and the rank statistics $\w F(Y_i)'s$. Second note that the rank statistics based on the $Y_i's$ are equal to the rank statistics based on the uniformized variables $F(Y_i)$'s. As a consequence of this two facts, the processes $\w c_{\w F_{\t{}}}$ and $\w C_{\w F_{\t{}}} $ can be constructed identically with the samples $(Z_i,Y_i)_{1\leq i\leq n}$ and $(Z_i,F(Y_i))_{1\leq i\leq n}$ . From now on in the proofs, since $F(Y_i)$ is uniformly distributed on $[0,1]$ (because of the continuity of $F$), we can assume without any loss of generality that the variable $Y$ is uniformly distributed.
\end{remark}

To compute the asymptotic distribution, since $c_F = c_{id}\circ F^{-}$,  we use the Delta method in metric spaces stated in Theorem 3.9.4 of \cite{vandervaart1996}. This approach has been employed for instance in \cite{vandervaart1996}, page 389, and in \cite{fermanian2004}, both in the context of the weak convergence of the empirical copula process. More precisely, we follow this scheme:

\begin{enumerate}[i)]
\item \label{step1} Use Lemma \ref{lemproduct} to obtain the weak convergence of the process $(s,t)\mapsto n^{1/2} (\w F(s)-F(s),$ $\w c_{id}(t)-c_{id}(t) )$.
\item \label{step2} Apply the Delta method with the map $(F,c_{id}) \mapsto c_{id}\circ F^{-}$.
\end{enumerate}
Because the latter map involves the quantile transformation that is not Hadamard differentiable everywhere (see Lemma 3.9.23 in \cite{vandervaart1996}), the fact that $\w F$ can be assumed to converge to the cdf of a uniform distribution (by Remark \ref{remranktransform}) is a key step in our proof. We define the function $\gamma_2:[0,1]^2 \r \R^{(p+1)\times (p+1)}$ given by
\begin{align*}
\gamma_2(u,v)= \cov\left(\begin{pmatrix}
1\\Z
\end{pmatrix} 1 _{\{Y\leq F^-(u)\}},\begin{pmatrix}
1\\Z
\end{pmatrix}\mathds 1 _{\{Y \leq F^-(v)\}} \right),
\end{align*}
and $\gamma_3:[0,1]^2 \r \R^{p\times p}$ by
\begin{align*}
\gamma_3(u,v) = (-\partial c_{F}(u),  I) \ \gamma_2(u,v)\ (- \partial c_{F}(v),  I) ^T,
\end{align*}
where $\partial c_{F}$ stands for the derivative of the map $u\mapsto c_F(u)$.


\begin{theorem}\label{th2}
Assume that $E[|Z_1|_2^2]$ is finite and $F$ is continuous. Then if $c_{F}$ is continuously differentiable, $\sqrt n(\w c_{\w F}- c_{F})$ converges weakly in $l^{\infty}([0,1])^p$ to a tight Gaussian process with zero-mean and covariance function $\gamma_3$.\end{theorem}

\begin{proof}
 Without loss of generality, we can put $F(Y_i)$ in place of $Y_i$ (see Remark \ref{remranktransform}). We denote by $id_{[0,1]}$ the cdf of the uniform distribution. By applying Lemma \ref{lemproduct}, the process $\sqrt n (\w G-G)$, with $\w G(s,t) =(\w F(s),\w c_{F}(t))$ and $ G(s,t) =( id_{[0,1]} (s),c_{F}(t))$, converges weakly in $l^{\infty} (\R)\times l^{\infty} ([0,1])^p$ to a tight Gaussian element. Now since
\begin{align}
\w c_{\w F}=\psi(\w G) \qquad\text{and}\qquad c_{F}=\psi(G)  ,
\end{align}
where $\psi: \mathcal F\times l^{\infty} ([0,1])^p \r l^{\infty} ([0,1])^p$, $\mathcal F$ being the space of cdf with support included in $[0,1]$, is given by
\begin{align}\label{decompmap}
\psi:(f_1,f_2)\mapsto(f_1^-,f_2) \mapsto f_2 \circ f_1^{-},
\end{align}
we can apply Theorem 3.9.4, page 374 in \cite{vandervaart1996} which basically says that $\sqrt n (\psi(\w G)-\psi( G))$ is $P$-Donsker provided that the map $\psi$ is Hadamard differentiable. In what follows, we first show that $\psi$ is Hadamard differentiable, and then we compute the asymptotic variance. Using Lemma 3.9.23, assertion (ii), page 386 in \cite{vandervaart1996}, the first map of Equation (\ref{decompmap}) reduced to $f\mapsto f^-$ is Hadamard differentiable at the function $id_{[0,1]}$ tangentially to $C[0,1]$. Moreover its derivative at $id_{[0,1]}$, in the direction $h_1$ is given by $  -h_1 $. Since $c_{F} $ is Fr\'echet differentiable, by Lemma 3.9.27, page 388 in \cite{vandervaart1996}, the second map in Equation (\ref{decompmap}) is Hadamard differentiable at $(id_{[0,1]}^-,c_{F})$, tangentially to $C[0,1]$ (because continuous functions are uniformly continuous on compacts). Its derivative at $(id_{[0,1]}^-,c_{F})$, in the direction $(h_1,h_2)$, is given by $ h_1\times  \partial c_{F}  +h_2$. By the chain rule, the function $\psi$ is Hadamard differentiable at the point $(id_{[0,1]},  c_{F})$ tangentially to $C[0,1]$. At this point, in the direction $(h_1,h_2)$, its derivative is given by $ -h_1\times \partial c_{F}  + h_2$. Hence, the limiting process has the representation
\begin{align*}
u\mapsto  w_1 (u)- \partial c_{F} (u) \times B(u)=(-\partial c_{F}(u),  I) \begin{pmatrix}B(u) \\w_1(u)\end{pmatrix},
\end{align*}
where $(B,w_1)$ is the Gaussian limit of $\w H:u\mapsto \sqrt n (\w G-G)\circ (u,u)$. Its covariance function is computed by applying the central limit theorem that gives
\begin{align*}
(\w H (u_1),\ldots ,\w H (u_K))  \overset{\t{d}}{\lr} ((B,w_1)(u_1),\ldots,(B,w_1)(u_K)),
\end{align*}
where $\vecv((B,w_1)(u_1),...,(B,w_1)(u_K))$ is a Gaussian vector with mean $0$ and covariance matrix having the block decomposition $(\gamma_2(u_k,u_l) )_{1\leq k,l\leq K}$.


\end{proof}

To obtain a similar result about the order $2$ moments process, we define the function $\Gamma_2:[0,1]^2 \r \R^{(p+1)\times (p+1)}$ by
\begin{align*}
\Gamma_2(u,v)= \cov\left(\begin{pmatrix}
1\\ \vecv(ZZ^T-I)
\end{pmatrix} 1 _{\{Y\leq F^{-1}(u)\}},\begin{pmatrix}
1\\ \vecv(ZZ^T-I)
\end{pmatrix}\mathds 1 _{\{Y\leq F^{-1}(v)\}} \right),
\end{align*}
and $\Gamma_3:[0,1]^2 \r \R^{p\times p}$ by
\begin{align*}
\Gamma_3(u,v) = (-\partial \t{vec}(C_{F})(u), I) \ \Gamma_2(u,v)\ (-\partial\t{vec}(C_{F})(v),  I) ^T.
\end{align*}
where $\partial \t{vec}(C_{F})(u)$ stands for the derivative of the map $u\mapsto \t{vec}(C_{F})(u)$.

\begin{corollary}\label{cor2}
Assume that $E[|Z_1|_2^4]$ is finite and $F$ is continuous. Then if $\t{vec}(C_{F})$ is continuously differentiable, $\sqrt n(\w C_{\w F}- C_{F})$ converges weakly in $l^{\infty}([0,1])^{(p\times p)}$ to a tight Gaussian process with zero-mean and covariance function $\Gamma_3$.
\end{corollary}
\begin{proof}
We apply Theorem \ref{th1} with $\text{vec}(ZZ^T-I)$ in place of $Z$.
\end{proof}

\subsection{The Bootstrap}\label{s24}

In light of the limiting covariance processes given in the previous section, in particular because of the presence of $\partial c_{F}$ and $\partial \t{vec}(C_{F})$ but also the possibly high-dimensionality of these processes, the asymptotic distributions are fairly hard to estimate. As a consequence, for making inference, it seems necessary to develop a bootstrap strategy. Efron \cite{efron1979} introduced the original bootstrap that consists in a sampling with equi-probability and replacement of the original sample. In \cite{wellner1993b}, the authors considered a more general re-sampling plan based on weights $w_{i,n}$, $i=1,\ldots,n$ that verified
\begin{enumerate}[(\t{B}1)]
\item\label{Bweightfirst} the random sequence $(w_{i,n})_{1\leq i\leq n}$ is exchangeable, i.e. for every permutation $(\pi_1,\ldots, \pi_n)$ of $(1,\ldots ,n )$, $(w_{i,n})_{1\leq i\leq n}$ has the same law as $(w_{\pi_i,n})_{1\leq i\leq n}$,
\item denote by $S_n$ the survival function of $w_{1,n}$, we have
\begin{align*}
\sup_{n\geq 1} \int S_{n}(u)^{1/2} du<+\infty \qquad \t{and} \qquad \lim_{A\r +\infty} \limsup_{n\r +\infty} \sup_{t\geq A}\ t^{2} S_{n}(t)=0.
\end{align*}
\item \label{Bweightend} $w_{i,n}\geq 0$, $\sum_{i=1}^n w_{i,n}=n$, $n^{-1}\sum_{i=1}^n (w_{i,n}-1)^2 \overset{P}{\r} 1$.
\end{enumerate}
Examples of such weights, are given in \cite{wellner1993b}. Now we define the bootstrap processes
\begin{align*}
\w F^*(t)  &= n^{-1}\sum_{i=1}^n w_{i,n}\mathds 1 _{\{Y_i\leq t\}},\\
\w c_{\Phi} ^*(u) &= \frac 1 n \sum_{i=1}^n  w_{i,n} Z_i \mathds 1 _{\{Y_i\leq \Phi_{\t{}}^-(u)\}},
\end{align*}
for every $t\in \R$ and every $u\in [0,1]$. The bootstrap of $\w c_{\Phi}$ (resp. $\w c_{\w F}$) is made by $\w c_{\Phi} ^*$ (resp. $\w c_{\w F^*}^*$). The following theorem basically says that the bootstrap in probability (in the sense of \cite{wellner1993b}) works.

\begin{theorem}\label{thboot}
Under (B\ref{Bweightfirst}) to (B\ref{Bweightend}), assume that $E[|Z_1|_2^2]$ is finite and $\Phi$ is a cdf, then conditionally on the sample,

\noindent \hspace{1cm} $ n^{1/2}(\w c^* _{\Phi}- \w c_{\Phi})$ has the same weak limit as $n^{1/2}(\w c_{\Phi}- c_{ \Phi})$, in probability.

\noindent If moreover $F$ is continuous and $c_F$ is continuously differentiable, then conditionally on the sample,

\noindent \hspace{1cm}$n^{1/2}(\w c^* _{\w F^*}- \w c_{\w F})$ has the same weak limit as $n^{1/2}(\w c_{\w F}- c_{ F})$, in probability.\end{theorem}

\begin{proof}

The first statement is a direct consequence of Lemma \ref{lemproduct} and Theorem 2.1 in \cite{wellner1993b}. For the second statement, we first apply the trick detailed in Remark \ref{remranktransform} to the bootstrap estimator.  Indeed it is easy to see that $\w c_{\w F^*}^*$ can be constructed as well from the sample $(Z_i,F(Y_i))_{1\leq i\leq n}$, so that the weak limit of $\w F^*$ can be assumed to be $id_{[0,1]}$. This is due to the equivalence between
\begin{align*}
\{Y_i\leq  \w F^{*-}(u) \} \Leftrightarrow \{\w F^* (Y_i)< u+n^{-1} w_{i,n} \},
\end{align*} 
for any $u\in [0,1]$, plus the fact that the bootstrap ranks $\w F^* (Y_i)$'s are the same as the uniformized bootstrap ranks (i.e. based on the $F(Y_i)$'s rather than the $Y_i$'s). Then by applying again Lemma \ref{lemproduct} with Theorem 2.1 in Paestgrad and Wellner, the process $\sqrt n (\w G^*-\w G)$, with $ \w G^*(s,t)= (\w F^*(s),\w c_{F}^*(t))$, has the same limiting distribution as $\sqrt n (\w G-G)$ (defined in the proof of Theorem \ref{th2}), that is a tight Gaussian element of $l^{\infty} (\R)\times l^{\infty} ([0,1])^p$. Then we can invoke the Delta-method for the bootstrap stated as Theorem 3.9.11, page 378, in \cite{vandervaart1996}.
\end{proof}

Similarly, we define $\w C_{\Phi} ^*$ by
\begin{align*}
\w C_{\Phi} ^*(u) &= \frac 1 n \sum_{i=1}^n w_{i,n} (Z_i  Z_i ^T-I) \mathds 1 _{\{Y_i\leq \Phi^-(u)\}},
\end{align*}
for every $u\in [0,1]$, and we obtain this corollary.
\begin{corollary}\label{corboot}
Under (B\ref{Bweightfirst}) to (B\ref{Bweightend}), assume that $E[|Z_1|_2^4]$ is finite and $\Phi$ is a cdf, then conditionally on the sample,

\noindent \hspace{1cm} $\sqrt n(\w C^* _{\Phi}- \w C_{\Phi})$ has the same weak limit as $\sqrt n(\w C_{\Phi}- C_{ \Phi})$, in probability.

\noindent If moreover $F$ is continuous and $\t{vec}(C_{F})$ is continuously differentiable, then conditionally on the sample,

\noindent \hspace{1cm}$\sqrt n(\w C^* _{\w F^*}- \w C_{\w F})$ has the same weak limit as $\sqrt n(\w C_{\w F}- C_{ F})$, in probability.\end{corollary}

\begin{proof}
We apply Theorem \ref{th1} with $\text{vec}(ZZ^T-I)$ in place of $Z$.
\end{proof}

\section[Inverse regression]{Application to inverse regression}\label{s3}

In this section we are based on the results of the previous section in order to (i) raise some new points about the asymptotics of SIR, (ii) develop a unified framework for inverse regression and (iii) study new bootstrap testing procedure. The variables $Z_i$'s are assumed to be standardized in order to clarify the statements of the results. In practice we must account for the error induced by estimations of the mean and the variance (see Section \ref{s4} for more details).

\subsection{Revisiting sliced inverse regression}\label{s31}

Sliced inverse regression \cite{li1991} is based on the vectors
\begin{align*}
n^{-1}\sum_{i=1}^n Z_i\mathds 1_{\{Y_i\in I(h)\}},
\end{align*}
where $I(h) $, for $h=1,\ldots,H$ is a partition of the range of the $Y_i$'s. In practice, to diminish the chance of having a poor estimation of such vectors, it is convenient to keep the same number of observations within each slice (this was already pointed-out in Remark 4.2 by \cite{li1991} and this is how SIR is usually run). Consequently each member $I(h)$ of the partition is random because it depends on the $Y_i$'s. Meanwhile when describing the asymptotic behaviour, many authors have ignored this additional source of randomness (see among others \cite{cook2004}, \cite{cook2005} or \cite{portier2013}). In what follows, we show that the randomness of the partition $I(h)$ can not be neglected since we find that it participates in the asymptotic variance of the estimation. Our approach can work because the slicing $I(h)$ is expressed in a simple way with the help of the rank statistics $\w F(Y_i)$'s. Hence we shall apply in the next, Theorem \ref{th2} and Corollary \ref{cor2}. For brevity, we focus on SIR, but the same analysis can be extended to second order slicing methods such as for instance, SAVE and DR.


Consider a multi-slice procedure with $H$ slices. Denote by $\lceil \alpha \rceil$ the smallest integer greater than or equal to $\alpha$. A reasonable way to dispatch the data among the slices should be with $\lceil n/H \rceil$ observations in the first slice, $\lceil 2n/H\rceil-\lceil n/H\rceil$ in the second,...,$n-\lceil n(H-1)/H\rceil$ in the last slice. Note that as soon as $n$ is a multiple of $H$, each slice contains exactly the same number of observations $n/H$. The SIR estimator is the subspace generated by
\begin{align*}
( \w c_{\w F}(u_1),\w c_{\w F}(u_2)-\w c_{\w F}(u_1),\ldots,\w c_{\w F}(u_{H})-\w c_{\w F}(u_{H-1})),
\end{align*}
with $u_h =  h/H$, and the corresponding estimation with nonrandom slices is the span of the matrix
\begin{align*}
 ( \w c_{ F}(u_1),\w c_{F}(u_2)-\w c_{F}(u_1),\ldots,\w c_{ F}(u_{H})-\w c_{ F}(u_{H-1})).
\end{align*}
Invoking Theorems \ref{th1} and \ref{th2}, for any $h\in\{1,\ldots, H\}$, because the sequences $\w c_{\w F}(u_h)$ and $\w c_{F}(u_h)$ have a different asymptotic distribution, the latter matrices neither. To highlight differences in the behaviour of $\w c_{\w F }$ and $\w c_{F }$, we consider the following tool model
\begin{align}\label{toolmodel}
Y=X_1+.1e,
\end{align}
where $(X,e)\in \R^{5} $ follows a standard normal distribution. In order to keep clear our statements and conclusions, we focus on the first slice of SIR in which the number of observations $\lceil nu \rceil $ varies with $u$ from $1/2$ to $0$. There are two different ways to compute it: 
\begin{itemize}
\item Order the responses $Y_i$'s, create a slice containing the first $\lceil nu \rceil $ observations, compute the mean over the $X_i$'s within the slice. This gives the vector  $\w c_{1}=\w c_{\w F}(u)$.
\item Create a slice according to $Y_i\leq F^-(u)$ (the slice is independent of the observations $Y_i$'s), compute the mean over the $X_i$'s within the slice. This gives the vector $\w c_{2}=\w c_{F}(u)$.
\end{itemize}
By means of simulations, we evaluate the first coordinate of the latter quantities $1000$ times. The resulting boxplots, for different values of $n$ are reported in Figure \ref{fig1}.

\begin{figure}\centering
\includegraphics[width=7.5cm,height=6cm]{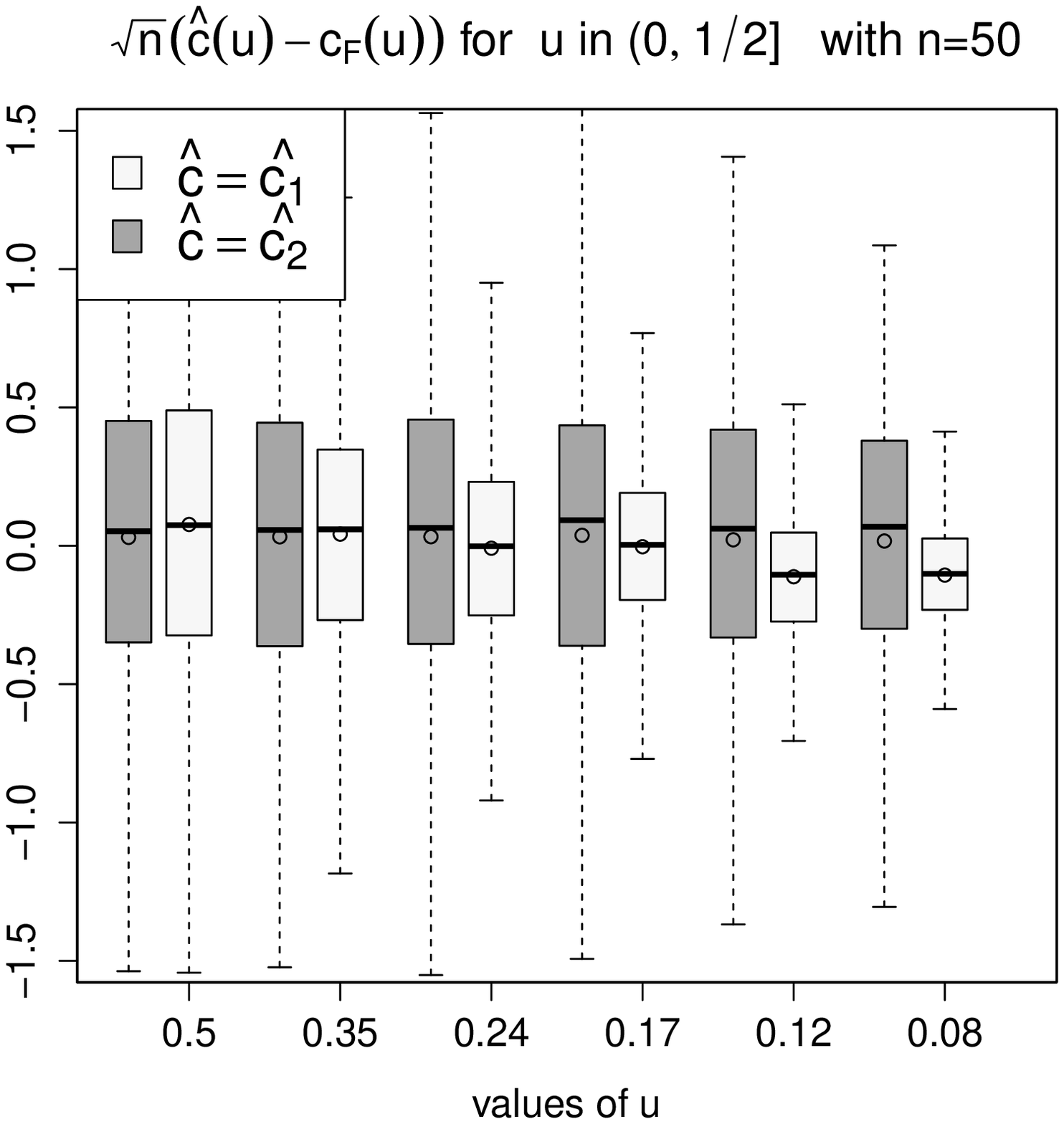} 
\includegraphics[width=7.5cm,height=6cm]{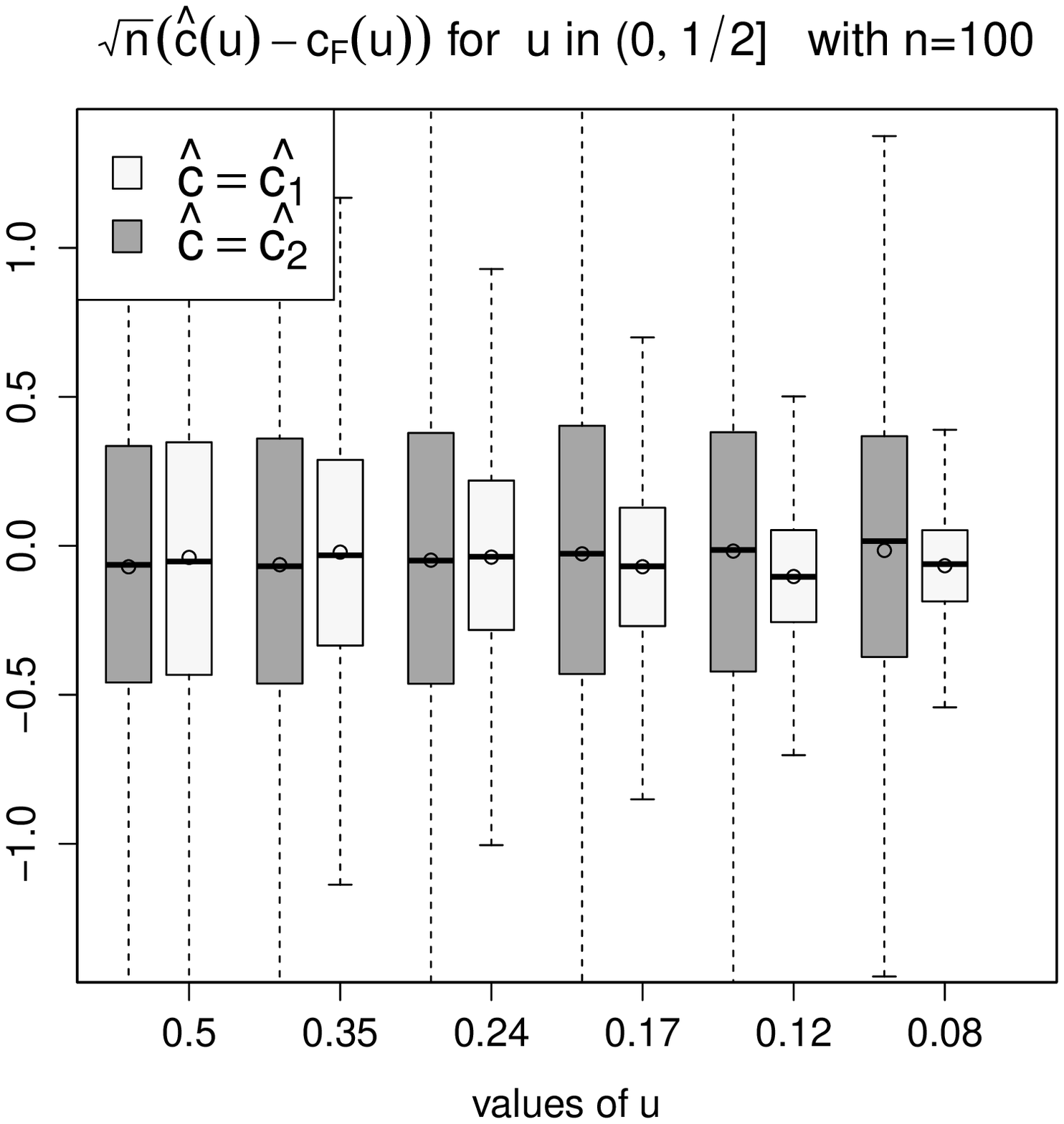} 
\includegraphics[width=7.5cm,height=6cm]{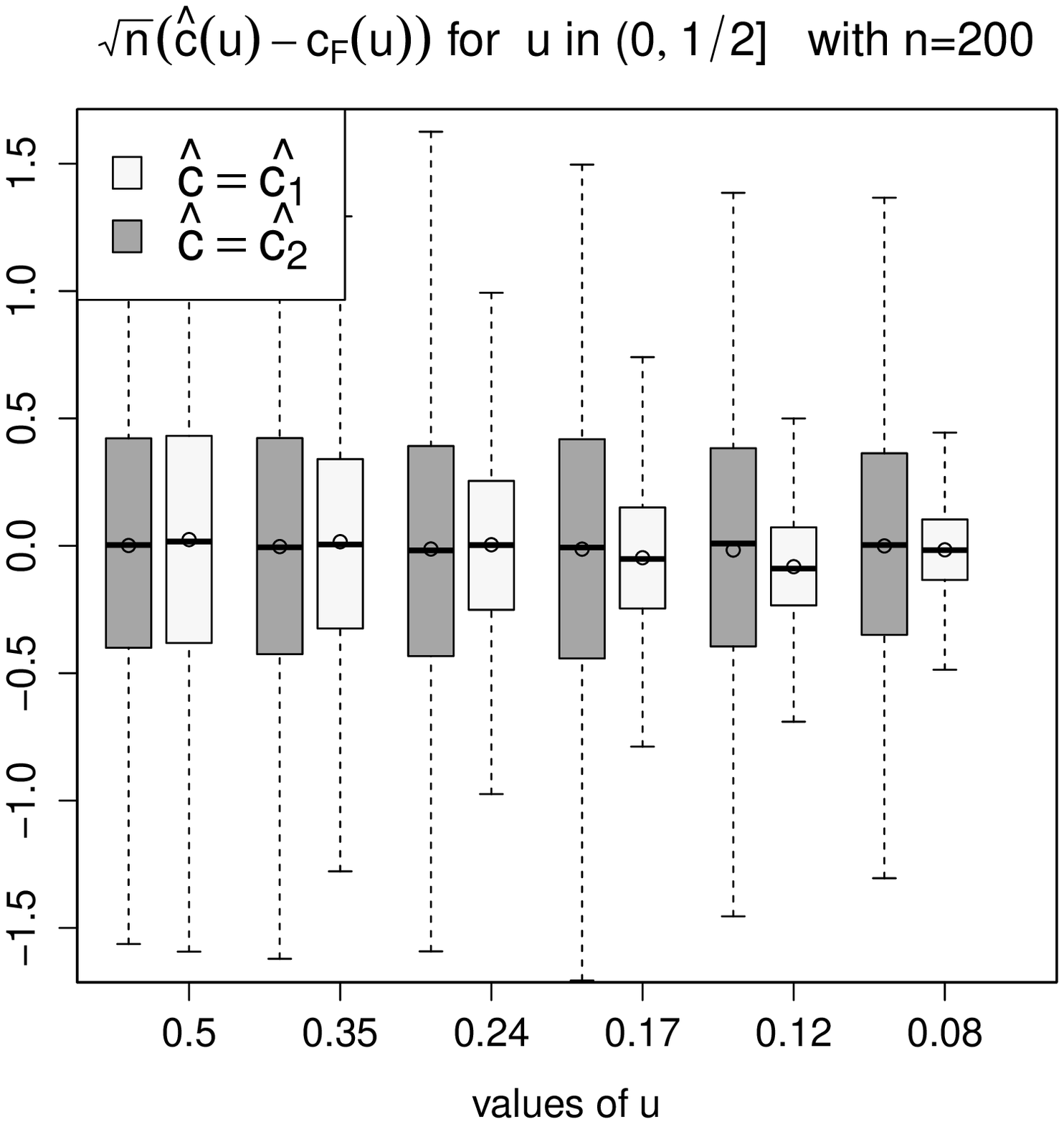} 
\includegraphics[width=7.5cm,height=6cm]{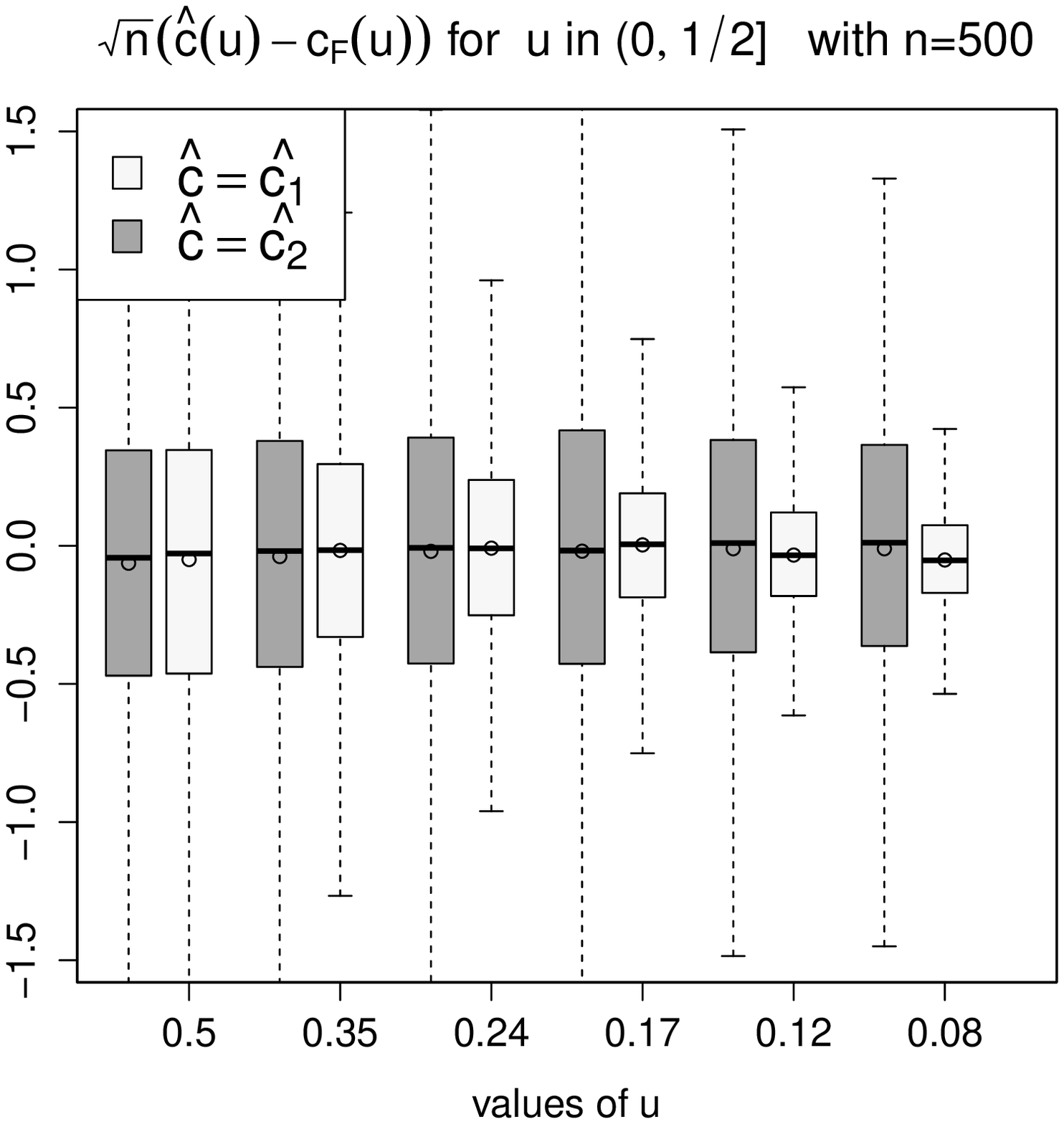} 
\caption{Boxplots of the law of the first coordinate of $\sqrt n (\w c_{k}(u) -c_F(u))$, for $k=1,2$, based on 1000 replications.}
\label{fig1}
\end{figure}

Starting from $u=1/2$ (meaning that the observations have been cut in half), where both variances are the same, we see that, as $u$ decreases, the dispersion of $\w c_{2}(u)$ becomes larger, whereas it is clearly more stable for $\w c_{1}(u)$. Note also that whereas $\w c_{2}(u)$ is unbiased, $\w c_{1}(u)$ suffers from a slight bias in small sample sizes. This sheds light on two things: (i) the limiting distribution of $\w c_{1}(u)$ and $\w c_{2}(u)$ are different and one should care about that as soon as inference is of matter, (ii) for this generic example, $\w c_{1}(u)$ is more efficient than $\w c_{2}(u)$, highlighting that it is more accurate to have a control on the number of observations within the slices. To the best of our knowledge, the latter question is still open although this is only theoretical because the matrix based on $c_2$ is not even computable (unless we know the law of $Y$).


\begin{remark}\label{remregressionexample}\normalfont 
In light of Theorems \ref{th1} and \ref{th2}, it is the function $\partial c_F $  that determines whether the asymptotic is affected by the randomness of the slices. In the case of an additive regression model $Y=g(\beta_0^TX)+e$ with $e\indep X$, we find that
\begin{align*}
c_F (u) = E[ZF_{e} (F^{-}(u) -g(\beta_0^TX))],
\end{align*}
where $F_e$ is the cdf of $e$. At $u=0$ and $u=1$ this quantity equals $0$, then under the assumption of Theorem \ref{th2}, by the Rolle's theorem there exists at least one $v\in (0,1)$ such that $\partial c_F(v)=0$. As a consequence, the asymptotic distributions of $\w c_{\w F}(v)$ and $\w c_{ F}(v)$ are the same. Nevertheless this certainly will not happen at each slice boundary $u_h$, as it is highlighted in Figure \ref{fig1} for Model (\ref{toolmodel}), for which $v=1/2$.
\end{remark}

\begin{remark}[cumulative slicing estimation]\label{remcomputationintegral1}\normalfont
In \cite{zhu2010}, the authors consider spaces generated by the integral
\begin{align*}
\int_0^1  \w c_{\w F}(u) \w c_{\w F}(u)^T d u.
\end{align*}
Contrary to most of the existing methods, any slicing is no longer necessary. They focus on a large $p$ small $n$ context and give a limit theorem by borrowing a $U$-statistic approach, that is rather different than our empirical process approach. Their simulation results highlight that CUME is competitive with SIR and performs even better in several situations.

\end{remark}

\subsection{Integral approach: a unified framework}\label{s32}

In order to consider in the next a broad class of different methods, and in particular to include SIR and CUME, it is useful to introduce the matrix
\begin{align*}
  A_{\nu}(\w \mu) = \int  \w \mu (u) \w \mu (u)^T d \nu (u),
\end{align*}
where $\w \mu $ belongs to the space $l^{\infty}([0,1])^{(p\times q)}$ with $q\geq 1$ and $\nu$ is a probability measure on $[0,1]$. In order to estimate $E_c$, the process $\w \mu $ shall be a combination of processes studied in the Section \ref{s2}, namely $\w c_{\w F}$ and $\w C_{\w F}$. As soon as (LC) and (CCV) are realized, the corresponding limit of  $A_{\nu}(\w \mu) $ generates a subspace of $E_c$. Hence the estimation of $E_c$ follows from an eigendecomposition of $A_{\nu}(\w \mu)$ by taking the eigenvectors associated to the $d_0$ largest eigenvalues as the estimated basis of $E_c$.

\paragraph{Order $1$ moments based methods.}
One easily sees that taking $\w \mu $ equal to $\w c_{\w F}(u)$ and $\nu$ equal to the uniform distribution on $[0,1]$, leads to CUME. Now let $\pi\in[0,1]$ and define the quantities
\begin{align*}
&\w m(u,\pi ) = \w c_{\w F}(u+\pi/2)-\w c_{\w F}(u-\pi/2),\\
&m(u,\pi ) =  c_{F}(u+\pi/2)- c_{F}(u-\pi/2),
\end{align*}
the SIR estimators can be expressed as the space generated by
\begin{align*}
 \sum_{h=1}^H \w m((2h-1)/2H,1/H ) \w m((2h-1)/(2H),1/H )^T.
\end{align*}
This is a direct consequence of the definition of SIR given in the previous section. As a result, SIR with $H$ slices bolongs to our framework. It corresponds to the matrix $ \w A_{\nu}(\mu) $ when $\w \mu$ equal to $\w m( \cdot ,1/H )$ and $\nu$ is the cdf of a discrete uniform random variable over the set $\{(2h-1)/(2H),\ h=1,\ldots, H\}$. Our framework permits also a slight modification of SIR, based on the same process $\w m( \cdot,\pi )$ but with $\nu$ being the cdf of a continuous (rather than discrete) uniform random variable on $[0,1]$. 

\paragraph{Order $2$ moments based methods.} 
As it is well-known in the literature, SIR and CUME are inconsistent in estimating directions that present a symmetric relationship with the variable $Y$. To remedy this problem, one can rather consider order $2$ moments of the predictor as in SAVE or DR. Within our framework, this means computing $A_\nu(\w \mu) $ with $\w\mu$ equal to $ \w C_{\w F}$ (cumulative version) or $\w M (u,\pi)=\w C_{\w F}(u+\pi/2)-\w C_{\w F}(u-\pi/2)$ (slicing version).

\bigskip

Denoting by $\mu$ the limit in probability of $\w \mu$, we obtain the weak convergence of $n^{1/2} ( A_{\nu}(\w \mu)-A_{\nu}(\mu))$ by following these steps:
\begin{enumerate}[(A)]
\item \label{ga}Weak convergence of the process $ n^{1/2} (\w \mu-  \mu)$ in $l^{\infty}([0,1])^{(p\times q)}$.
\item \label{gb} Application of the continuous mapping theorem to extend the convergence to some integral maps.
\end{enumerate}
In Section \ref{s2}, we have focused on the first step, so that the proof of the following theorem essentially consists in showing the second step. For brevity we formally state our results for the order $1$ moments based methods, the extension to order $2$ moments based methods being straightforward (see Remark \ref{remarkcomputationintegral2}).

\begin{theorem}\label{thconvmatrix}
Assume that $E[|Z_1|_2^2]$ is finite, $F$ is continuous and $c_{F}$ is continuously differentiable, then 

\noindent (i) if $\w \mu= \w c_{\w F}$, $\mu= c_{ F}$, 
\begin{align*}
n^{1/2}( A_{\nu}(\w \mu)-A_{\nu}(\mu)) \text{\quad has Gaussian limit\quad }\int 
 w_3(u)c_{ F}^T+c_{ F}w_3(u)^T d \nu(u),
\end{align*}

\noindent (ii) if $\pi\in [0,1]$, $\w \mu= \w m(\cdot,\pi)$, $ \mu=  m(\cdot ,\pi)$, 
\begin{align*}
n^{1/2}( A_{\nu}(\w \mu)-A_{\nu}(\mu)) \text{\quad has Gaussian limit\quad }\int 
 w_4(u)m(u,\pi)^T+m(u,\pi)w_4(u)^T d \nu(u),
\end{align*}
where $w_4(u)=w_3(u+\pi/2)-w_3(u-\pi/2)$ and $w_3$ is a Gaussian process with covariance $\gamma_3$.
\end{theorem}


\begin{proof}
Since the proofs of (i) and (ii) are very similar we focus on (ii). Invoking Theorem \ref{th1} and the continuous mapping theorem stated for instance in \cite{vandervaart1996}, page 20, as Theorem 1.3.6,  we obtain the weak convergence of $\{\sqrt n (\w m(u,\pi) - m(u,\pi))\}_{u\in [0,1]}$ to the Gaussian process $w_4$. For every $u\in [0,1]$ and $\pi\in [0,1]$, one can write
\begin{align*}
&\w m(u,\pi) \w m(u,\pi)^T - m(u,\pi) m(u,\pi)^T \\
 &= \big(\w m(u,\pi) -  m(u,\pi)\big) m(u,\pi)^T + m(u,\pi)\big(\w m(u,\pi)- m(u,\pi)\big)^T \\
&\hspace{3cm}+\big(\w m(u,\pi)- m(u,\pi)\big)\big(\w m(u,\pi)- m(u,\pi)\big)^T,
\end{align*}
then, as a consequence of the Delta-method, $\{\sqrt n (\w m(u,\pi) \w m(u,\pi)^T - m(u,\pi) m(u,\pi))^T\}_{u\in [0,1]} $ converges weakly to $\{w_4(u) m(u,\pi)^T+m(u,\pi) w_4(u)^T\}_{u\in [0,1]}$. Finally applying the continuous mapping theorem to the previous process with the map $f\mapsto \int f(u)d\nu(u)$ , we obtain the statement of the theorem.
\end{proof}

\begin{remark}[coverage property]\label{remarkcoveringtheCS}\normalfont

A comparison between the spaces generated by CUME and SIR is relevant to highlight the differences between continuous and discrete methods. The space that SIR estimates is
\begin{align*}
E_{\t{SIR}}^H = \spann \{ m((2h-1)/2H,1/H ),\  h=1,\ldots,H \},
\end{align*}
and under the conditions of Theorem 3 in \cite{portier2013}, for $H$ sufficiently large, $E_{\t {SIR}}^H=E_c$. This result is important because it ensures that when $H$ increases, SIR eventually estimates the whole subspace. Nevertheless, this is not sufficient to guarantee a complete estimation of $E_c$ since in practice, we do not know how to choose $H$. The space estimated by CUME is
\begin{align*}
E_{\t{CUME}} = \spann \{c_F(u ) , \ u\in[0,1]\}.
\end{align*}
It follows that $E_{\t{SIR}}^H\subset  E_{\t{CUME}}   \subset E_c$. As a consequence, compared with SIR, the method CUME is more likely to recover a larger subspace within $E_c$.
\end{remark}

The bootstrap is made through
\begin{align*}
  A_{\nu}(\w \mu ^* ) = \int  \w \mu^*(u)\w \mu^*(u)^T d \nu (u),
\end{align*}
where $\w\mu^*$ is a bootstrap version of $\mu$ that can be chosen according to the next theorem. We define the bootstrap process $\w m^*(u,\pi) = \w c^*_{\w F^*}(u+\pi/2)-\w c^*_{\w F^*}(u-\pi/2)$.

\begin{theorem}\label{thbootmatrix}
Under (B\ref{Bweightfirst}) to (B\ref{Bweightend}), assume that $F$ is continuous and $c_{F}$ is continuously differentiable, then conditionally on the sample,
\begin{align*}
n^{1/2}(  A_{\nu}(\w \mu^*)-  A_{\nu}(\w \mu )) \t{\quad has the same weak limit as\quad } n^{1/2}( A_{\nu}(\w \mu)-A_{\nu}(\mu))  \t{, in probability,}
\end{align*}
provided that (i) $\w \mu ^* =  \w c^*_{\w F^*}$ and $\w \mu  =  \w c_{\w F}$ or (ii)  $\w \mu ^* =  \w m^*(\cdot,\pi)$ and $\w \mu  =  \w m(\cdot,\pi)$.
\end{theorem}

\begin{proof}
The proof is similar as the proof of Theorem \ref{thconvmatrix} with the following changes: consider the probability space conditional on the $(Y_i,Z_i)$'s and replace $\w m(u,\pi)$ by $\w m^*(u,\pi)$ and $ m(u,\pi)$ by $\w m(u,\pi)$. 
\end{proof}

\begin{remark}[bootstrapping the slices]\label{remarkbootstrap}\normalfont
An accurate description of the asymptotic distribution is necessary for making precise the inference. By Theorem \ref{thbootmatrix}, our bootstrap procedure is valid and therefore, shall be use to make inference on $ A_{\nu}(\w\mu)$. This is mainly due to the fact that the randomness of the slices has been reproduced by bootstrapping also the estimated cdf of $Y$, e.g. contrary to $\w c^*_{\w F^*}$, the process $\w c^*_{\w F}$ won't produce a valid bootstrap. Hence, other bootstrap techniques that ignore this randomness will fail in bootstrapping the law of $ A_{\nu}(\w \mu)$. Existing bootstrap methods for SIR (\cite{velilla2007}, \cite{portier2013}) consider the slices as fix, and so they are unable to reproduce correctly the law of SIR as it is usually computed. Nevertheless, when testing specific properties of $E_c$, it could happen that both bootstrap, respectively directed by $\w c^*_{\w F^*}$ and $\w c^*_{\w F}$, work (see Section \ref{s33} for more details).

\end{remark}

\begin{remark}[order $2$ moments based methods]\normalfont\label{remarkcomputationintegral2}
Assuming that $E[|Z_1|_2^4]$ is finite, $F$ is continuous and $\t{vec}(C_{F})$ is continuously differentiable, it is an easy exercise to obtain a similar statement as in Theorem \ref{thconvmatrix} (replacing $w_3$ by $W_3$, invoking Corollary \ref{cor2} rather than Theorem \ref{th2}) as well as in Theorem \ref{thbootmatrix} (replacing $c$ by $C$, invoking Corollary \ref{corboot} rather than Theorem \ref{thboot}).
\end{remark}

\subsection{Cram\'er-von Mises tests}\label{s33}

The integral methods of the previous section, such as SIR and CUME, produce accurate estimations of $E_c$ (see for instance the simulation study in \cite{zhu2010}). Nevertheless, the asymptotic distribution of these methods was unknown from the researchers making difficult any inference based on the matrix $A_\nu (\w\mu)$. On the one hand, some authors neglected the effect of the randomness of the slices for SIR (\cite{cook2004}, \cite{cook2005} or \cite{portier2013}), on the other hand, other ones employed in addition the Bentler and Xie's approximation \cite{bentler2000} in order to compute the asymptotic distribution (see \cite{bura2011} and \cite{portier2014}). Here based on the empirical process approach of Section \ref{s2}, the purpose is to demonstrate rigorously that bootstrap leads to accurate inference when testing structural properties of $E_c$. We introduce three tests that asses: the dimension of $E_c$, the no effect of a set of predictors and the contribution of a given method. At the end of the section, we show
that all the tests considered are consistent and that bootstrap is valid to compute their quantiles. All the test statistics that we introduce are of the Cram\'er-von Mises type, i.e. of the form
\begin{align*}
\int |\w f(u) |_F^2 d\nu(u),
\end{align*}
where $\w f$ is a certain process that belongs to $l^{\infty} ([0,1])^{(p\times q)} $ and $|\cdot |_F $ is the Frobenius norm. In our precise situation, because the integrands are piecewise constant, closed-formulas are available, making the tests computationally feasible. This generally no longer happen for Kolmogorov type statistics.

\subsubsection{Testing dimensionality}

In order to determine the dimension $d_0$ of $E_c$, it is usual  to test whether $d_0$ equals a given number, say $d$, against the alternative $d_0$ is larger than $d$, i.e.
\begin{align}\label{testingdimension}
H_0: \ d_0=d\qquad  \t{against} \qquad H_1:\ d_0>d.
\end{align}
Then starting with $d=0$, if rejected we put $d:=d+1$, until the first acceptance. Different approaches that could be use are summarized in \cite{bura2011} and \cite{portier2014}. In the following we focus on the most common test statistic, based on the sum of eigenvalues of $ A_{\nu}(\w\mu)$, given by
\begin{align*}
\w \Lambda_1 =n \sum_{k=d+1}^p \w \lambda_k,
\end{align*}
where the $\w \lambda_k$'s are the eigenvalues of the matrix $A_{\nu}(\w\mu)$, arranged in decreasing order. We have the formula
\begin{align*}
\w \Lambda_1 &=n \t{trace}( \w QA_{\nu}(\w\mu) \w Q)=  n \int | \w Q \w \mu |_F^2d\nu(u), 
\end{align*}
where $\w Q$ is the eigenprojector on the eigenspace associated to the $p-d$ smallest eigenvalues of $A_{\nu}(\w\mu)$.

\subsubsection{Testing a predictor contribution}

 Following \cite{cook2004}, we develop tests of no effect, on the response variable $Y$, of a selected group of predictor, say $\eta^TZ$ where $\eta\in \R^{p\times (p-d)}$ is such that $\eta^T\eta=I$. We define $\beta$ such that $(\beta,\eta)\in \R^{p\times p} $ is an orthogonal matrix. We say that $\eta^TZ$ has no effect on $Y$ if
 \begin{align*}
P(Y\in A|\beta^TZ) =  P(Y\in A|Z),
 \end{align*}
for any Borel set $A$. By \cite{cook2004}, Proposition 1, this is equivalent to  $\eta \in E_c^{\perp}$. As a consequence, we introduce the hypotheses
\begin{align}\label{testpredictor}
H_0:\ \eta \in E_c^{\perp} \qquad \t{against} \qquad H_1:\ \eta \notin E_c^{\perp} .
\end{align}
Under the so-called coverage condition, that basically says that $E_c$ is spanned by $ A_{\nu}(\mu)$, the previous set of hypotheses is equivalent to
\begin{align*}
H_0:\  \eta^T A_{\nu}(\mu) \eta = 0  \quad \t{against} \quad H_1:\ \eta^T   A_{\nu}(\mu) \eta \neq  0.
\end{align*}
Therefore a natural statistic for testing $H_0$ is
\begin{align*}
\w \Lambda_2 = n\t{trace}( \eta^T   A_{\nu}(\w \mu)\eta)= n \int |\eta^T \w \mu (u)|_F^2d\nu(u).
\end{align*}

\subsubsection{Testing a method contribution}\label{s333}

Here we consider a given method whose estimated basis is noted $\w \beta\in \R^{p\times d}$. Let us assume that there exists a basis $\beta\in \R^{p\times d}$ such that $P_{\w \beta}$ converges in probability to $P_\beta$, with the notation $P_\beta=\beta\beta^T $. We want to test whether the method misses a direction (asymptotically), i.e. 
\begin{align}\label{testmethodcontribution}
H_0:\ \eta \in E_c^{\perp} \qquad \t{against} \qquad H_1:\ \eta \notin E_c^{\perp} ,
\end{align}
where $(\beta,\eta)\in \R^{p\times p} $ is an orthogonal matrix.
Let $\w \eta$ be such that $(\w \beta,\w \eta)$ is an orthogonal matrix, our statistic is given by
\begin{align*}
\w \Lambda_3 &=n\t{trace}( \w \eta^T  A_{\nu}(\w\mu)\w \eta)= n \int |\w \eta^T \w \mu(u)|_F^2d\nu(u).
\end{align*}
We have in mind two typical applications. First we aim at testing the so called SIR pathology, i.e. whether an order $1$ moments based method fails in recovering the whole subspace. For that purpose, $\w \beta$ might be for instance the estimated basis of SIR or CUME and $\w \mu$ should be based on the order $2$ process $\w C$. Clearly if the model is subject to the order $1$ pathology, the test shall reject $H_0$. Second the latter procedure can be applied to select the estimated directions for the order $2$ optimal function method introduced in \cite{portier2013}. This method alleviates the assumption CCV and produces accurate estimates but a classical eigenvalue-based selection of the directions fails. The initial test of independence developed in \cite{portier2013} rely on a null hypothesis that is too strong. It is more accurate to apply the above test when $\w \beta$ is the estimated basis of the order $2$ optimal function method.

\subsubsection{Consistency of the tests}\label{s334}

Theoretically, a test is said to be consistent if, as $n$ increase, the level converges to the nominal level and the power goes to $1$. As it will  be stressed out, every of the tests considered previously is consistent. Practically one needs to compute the quantiles of the asymptotic law of the statistic. In our case, those quantiles are difficult to estimate and this could diminish the accuracy of the test \cite{portier2014}. As a consequence, we recommend a bootstrap strategy for computing these quantiles and we show in the next  the consistency of our bootstrap procedure. 

For the sake of generality, we study all the tests (\ref{testingdimension}), (\ref{testpredictor}) and (\ref{testmethodcontribution}) introduced in the previous section. The statistics $\w \Lambda_k$ for $k=1,2,3$, can be written as follows
 \begin{align*}
\w \Lambda_k &= n \int |\w Q_k \w \mu(u)|_F^2d\nu(u) ,
\end{align*}
with $\w Q_1$ the  eigenprojector associated to the $p-d$ smallest eigenvalues of  $\int \w \mu(u)\w \mu(u) ^Td\nu(u)$, $\w Q_2 =\eta \eta^T$ with $\eta\in \R^{p\times(p-d)} $ a basis, and $\w Q_3=\w\eta\w\eta^T$ the orthogonal projector on the orthogonal complement of the estimated space of a given method as it is described in Section \ref{s333}. We also introduce (when they exist) $ Q_1$, the  eigenprojector associated to the $p-d$ smallest eigenvalues of  $\int  \mu(u) \mu(u) ^Td\nu(u)$, $ Q_2 =\eta \eta^T$, and $Q_3$, the limit of $\w Q_3$.

Bootstrap testing requires particular care so that the bootstrap estimator mimics the hypothesis $H_0$ even when $H_1$ is realized \cite{hall1991}, \cite{portier2014}. For statistics of a similar type as $\w \Lambda_1$, \cite{portier2014} shows that the quantiles can be computed using the technique of the constraint bootstrap. Following their approach, we define the bootstrap statistics $\w \Lambda_k^*$'s by
\begin{align*}
\w \Lambda^*_k = n \int |\w Q_k^* \w \mu^*_k(u)|_F^2 d\nu(u) \qquad \t{for } k=1,2,3,
\end{align*}
with $\w Q_1^*$ the eigenprojector associated to the $p-d$ smallest eigenvalues of $\int \w \mu^*_1(u)\w \mu_1^*(u) ^Td\nu(u)$, $\w Q_2^*= \eta \eta^T$,  $\w Q_3^*$ is a bootstrap version of $\w Q_3$, and for every $u\in [0,1]$,
\begin{align}\label{eqbasebootstrap}
\w \mu^*_k(u) = (I-\w Q_k) \w \mu(u) + (\w \mu^*(u)-\w \mu (u)) .
\end{align}
The latter formula is the cornerstone of the bootstrap procedure. It ensures that the bootstrap process $\w \mu_k^*$ is asymptotically contained in a subspace of dimension $d$, making the bootstrap process having a $H_0$-likely behaviour. To guarantee the consistency of the tests, we introduce the following assumptions. A discussion is postponed latter.

\begin{enumerate}[(\t{A}1)]
\item \label{A1}  The process $\mu:[0,1]\r \R^{p\times q}$ is continuous and $\spann (\mu(u),\ u\in [0,1])=  E_c$.
\item \label{A2} The process $ \w \mu :[0,1]\r \R^{p\times q}$ is such that
\begin{align*}
n^{1/2} (
\w \mu - \mu,
\w Q_k - Q_k) \t{ converges weakly in $l^{\infty}([0,1])^{(p\times q)}\times \R^{p\times p}$ to }   (w_\mu,w_Q).
\end{align*}
\item\label{A3}  The process $ \w \mu^* :[0,1]\r \R^{p\times q}$ is such that, conditionally on the sample,
\begin{align*}
n^{1/2} (
\w \mu^* -\w \mu,
\w Q^*_k -\w Q_k) \t{ converges weakly in $l^{\infty}([0,1])^{(p\times q)}\times \R^{p\times p}$ to } (\widetilde w_\mu,\widetilde w_Q), 
\end{align*}
in probability, with $(Q \widetilde w_\mu,\widetilde w_Q\mu )\overset {\t{d} }{=}(Q w_\mu, w_Q \mu)$. 
\end{enumerate}

The previous set of assumptions might be understood as follows. Assumption (A\ref{A1}) is the so called coverage condition that has been used by several authors \cite{cook2005}, \cite{portier2013}. This condition is discussed within the SIR and CUME context in Remark \ref{remarkcoveringtheCS}. Assumptions (A\ref{A2}) and (A\ref{A3}) depends on the test under consideration. When $k=2$, for SIR and CUME, (A\ref{A2}) (resp. (A\ref{A3})) is a straightforward consequence of Theorem \ref{th2} (resp. Theorem \ref{thboot}); for order $2$ moments based methods, it is implied by Corollary \ref{cor2} (resp. Corollary \ref{corboot}). The reader might refer to the mentioned theorems to obtain conditions that guarantee (A\ref{A2}) and (A\ref{A3}). For $k=1,3$, the theorems we just mentioned are not enough to obtain directly (A\ref{A2}) and (A\ref{A3}) because these conditions involve the joint distribution of the process $\w\mu$ with a certain eigenprojector. However they can be ascertained by the additional use of an asymptotic expansion for eigenprojectors e.g. Lemma 4.1 in \cite{tyler1981}. Finally, note that Assumption (A\ref{A3}) is weaker than asking for a complete bootstrap, i.e. that, conditionally on the sample, $n^{1/2} (
\w \mu^* -\w \mu,
\w Q^*_k -\w Q_k)$ has the same weak limit as $n^{1/2} (
\w \mu - \mu,
\w Q_k - Q_k)$, in probability. This will have interesting consequences on the validity of different bootstrap strategies (see the remark bellow).

\begin{proposition}\label{propbootstraptesting}
Under Assumptions (A\ref{A1}), (A\ref{A2}) and (A\ref{A3}), testing (\ref{testingdimension}), (\ref{testpredictor}) or (\ref{testmethodcontribution}) with respectively $\w \Lambda_1$, $\w \Lambda_2$, $\w \Lambda_3$ and calculation of the quantiles with $\w \Lambda_1^*$, $\w \Lambda_2^*$, $\w \Lambda_3^*$ respectively, is consistent.
\end{proposition}

\begin{proof}

Note that $\w \Lambda_k$ is a continuous transformation of the process $n^{1/2} \w Q_k \w \mu$. Under $H_0$, because (A\ref{A1}) implies that $Q_k\mu=0$, we have
 \begin{align}\label{decomplamba1}
 n^{1/2} \w Q_k \w \mu=n^{1/2}  Q_k( \w \mu-\mu)+ n^{1/2} (\w Q_k-Q_k)\mu + n^{1/2} (\w Q_k-Q_k)( \w \mu-\mu).
\end{align}
Using (A\ref{A2}), $\|\w \mu-\mu\|_\infty\r 0$ in probability, then by Slutsky's Lemma, the last term vanishes asymptotically. Using (A\ref{A3}) and the continuous mapping theorem, the sum of the first two terms in (\ref{decomplamba1}) (and so $n^{1/2} \w Q_k \w \mu$) converges weakly in $l^\infty ([0,1])^{(p\times q)}$. As a consequence of the continuous mapping theorem, under $H_0$, $\w\Lambda_k$ converges weakly to a real random variable. Under $H_1$, it is easy to show that $|Q_k\mu(u)|_2>0 $ for a certain $u\in[0,1]$, making $\w \Lambda_k$ going to infinity in probability. 

Consequently it is enough to show that the bootstrap statistic (i) has the same behaviour as the statistic under $H_0$, and (ii) remains bounded in probability under $H_1$. For (i), note that $\w \Lambda_k^*$ is a continuous transformation of the process  $n^{1/2} \w Q_k^* \w \mu_k^*$ that can be written as
 \begin{align*}
n^{1/2} \w Q_k \w \mu^*_k+n^{1/2} (\w Q_k^*-\w Q_k) \w \mu^*_k 
\end{align*}
then using the definition of  $\w \mu^*_k$, we get that
\begin{align*}
n^{1/2} \w Q_k^* \w \mu^*_k &= n^{1/2} \w Q_k (\w \mu^*-\w \mu) + n^{1/2} (\w Q_k^*-\w Q_k) (I-\w Q_k) \w \mu +n^{1/2} (\w Q_k^*-\w Q_k) (\w \mu^*-\w \mu ).
\end{align*}
The latter term is asymptotically neglectable by (A\ref{A3}), it follows that
\begin{align*}
n^{1/2} \w Q_k^* \w \mu^*_k &= n^{1/2} Q_k (\w \mu^*-\w \mu)+ n^{1/2} (\w Q_k^*-\w Q_k) (I-Q_k)  \mu
+ o_p(1).
\end{align*}
Since under $H_0$, $(I-Q_k)  \mu(u)= \mu(u)$, using (A\ref{A3}) and the continuous mapping theorem is enough to show that conditionally on the sample, $n^{1/2} \w Q_k^* \w \mu_k^*$ has the same asymptotic law as $n^{1/2} \w Q_k \w \mu_k$, in probability. Then invoking again the continuous mapping theorem provide the same conclusion with $\w\Lambda_k^*$ and $\w\Lambda_k$. Under $H_1$, in light of the latter representation and by (A\ref{A3}), conditionally on the sample, the sequence  $n^{1/2} \w Q_k^* \w \mu_k^*$ is tight.
\end{proof}


\begin{remark}[other bootstrap strategies]\normalfont\label{remarlbootstrapstrategy}
As we have highlighted (see Remarks \ref{remregressionexample} and \ref{remarkbootstrap}), the natural bootstrap candidate for $\w c_{\w F}$ is given by $\w c^*_{\w F^*}$ (rather than $\w c^*_{\w F}$), in which the estimated cdf $\w F$ has been bootstrapped. Because the randomness of the slices (carried by $\w F$) affects the limiting distribution, this can be seen, at first glance, as a necessary evil. In our particular context given by (\ref{testingdimension}), (\ref{testpredictor}) and (\ref{testmethodcontribution}), and under the linearity condition, it is in fact not essential to bootstrap $\w F$. Indeed, Assumptions (A\ref{A3}) only requires that the bootstrap estimator reproduces the law of $Q \sqrt n (\w c_{\w F} - c_{F})$ where $Q$ stands for the orthogonal projector on a given subspace of $E_c^\perp$. In light of the proof of Theorem \ref{th2}, we have that $\sqrt n (\w c_{\w F} - c_{F})$ has the following limiting distribution 
\begin{align*}
w_1-  \partial c_F B,
\end{align*}
where $(w_1,B)$ is a certain Gaussian process. Since for any $u\in[0,1]$, $\partial c_F(u) =E(Z| Y=F^-(u))   $, using the linearity condition we have that $\partial c_F\in E_c$. Multiplying by $Q$ the latter representation, we obtain that the asymptotic law of $Q \sqrt n (\w c_{\w F} - c_{F})$ is reduced to the representation $Q w_1$. As a consequence, the part $\partial c_F B$ in the asymptotic variance does not matter here, and so the bootstrap estimator given by $\w c^*_{\w F}$, satisfies assumption (A\ref{A3}) as well as $\w c^*_{\w F^*}$ does. Either for SIR or CUME, using $\w c^*_{\w F}$ is computationally less intensive than using $\w c^*_{\w F^*}$ because it preserves the slicing initially used for the estimator $\w c_{\w F}$. 
\end{remark}

\section{Simulations}\label{s4}

In this section, we study the accuracy of the bootstrap approximation facing one of the Cram\'er-von Mises tests introduced in Section \ref{s33}. We focus on the test of significance of some sets of predictors $\eta^TX$ described by (\ref{testpredictor}) and we consider the performance of both methods SIR and CUME with the statistic $\w \Lambda_2^*$. Our aim is to analyse quite difficult situations from small to moderate sample size.

Given i.i.d. observations from a regression model, we test whether a vector $\eta$ is orthogonal to $E_c$ or not. The statistics of interest are related to SIR with $H$ slices and CUME, each is given respectively by
\begin{align*}
&\w\Lambda_2^{\text{SIR}} =  n \int |\w Q_{\eta}  \w m (u,H^{-1}) |_F^2d\nu_d (u)\\
&\w\Lambda_2^{\text{CUME}} =  n \int |\w Q_{\eta} \w c_{\w F} (u) |_F^2d \nu_c (u),
\end{align*}
where $\w Q_{\eta}$ is the orthogonal projector on the space generated by $\w \Sigma^{-1/2}\eta$, $\w \Sigma$ is the classical estimator of the variance of $X$, and $\nu_d$ (resp. $\nu_c$) is the uniform probability measure on the set $\{(2h-1)/(2H),\ h=1,\ldots, H\}$ (resp. on the set $[0,1]$). The process $\w c_{\w F}$ and $\w m$ are the same as the ones define in the paper except that from now on, we estimate the mean and the variance of $X$.

The bootstrap estimators are computed following Equation (\ref{eqbasebootstrap}). As pointed out in Remark \ref{remarlbootstrapstrategy}, there are two different bootstrap strategies that are available to compute the quantiles of the test. The first one involves $\w c^*_{\w F^*}$ and gives, for instance, for CUME
\begin{align*}
 n \int |\w Q_{\eta}^* \{\w c^*_{\w F^*} (u) -\w Q_{\eta}\w c_{\w F} (u)\}|_F^2d \nu_c (u),
\end{align*}
where $\w Q_{\eta}^*$ is the orthogonal projector on the space generated by $\w \Sigma^{*-1/2}\eta$ and $\w \Sigma^*$ is defined in Remark \ref{remarkbootstrapsigma} bellow. This bootstrap is abbreviated in the next SIRb1 and CUMEb1. The second bootstrap involves $\w c^*_{\w F}$ and gives, for instance, for CUME
\begin{align*}
 n \int |\w Q_{\eta}^* \{\w c^*_{\w F} (u) -\w Q_{\eta}\w c_{\w F} (u)\} |_F^2d \nu_c (u),
\end{align*}
it is abbreviated by SIRb2 and CUMEb2. To compute a quantile of level $\alpha$, we draw independently $B$ bootstrap statistics and then calculate the empirical quantile of level $\alpha$ associated to this sample.

\begin{remark}[computation of CUME]\label{remarkcomputation}\normalfont
Either for the estimator or the bootstrap, integrals associated to CUME are computed easily because the integrands are piecewise constant. For instance, one may show that $\w\Lambda_2^{\text{CUME}} = n^{-1} \sum_{i=1}^n | \w Q_{\eta} \w c_{id} (Y_i)|_2^2$, and the same kind of formulas can be derived for the bootstrap statistics. Because, the integrand of the method b1 has $2n$ jumps whereas the integrand of the method b2 has $n$ jumps, the method b2 is less intensive computationally.

\end{remark}

\begin{remark}[standardizing the bootstrap]\label{remarkbootstrapsigma}\normalfont
For the sake of completeness, in this section we have leaved the theoretical framework of the paper that supposed to be known the mean and the variance of $X$. To build our estimators, we have plugged the classical estimators of the latter quantities in the initial estimators. This naturally induces an additional part in the asymptotic distribution. We account for this part by bootstrapping also the mean and the variance by respectively 
\begin{align*}
\overline{wX}= n^{-1}\sum_{i=1}^n w_{i,n}X_i \qquad \text{and}\qquad \w\Sigma^* = n^{-1}\sum_{i=1}^n w_{i,n}(X_i-\overline{wX})(X_i-\overline{wX})^T.
\end{align*} 
Note that they are used to standardized the predictors as well as to standardized the set of directions $\eta$ under test.
\end{remark}

We consider the following models:
\begin{align}
Y&=X_1 + \sigma e,\label{eqmodels1}\\
Y&=\frac{ X_1}{ 0.5+(2+X_2+X_3)^2} +\sigma e,\label{eqmodels2}\\
Y&=\exp(X_1) \times \sigma  e,\label{eqmodels3}
\end{align}
where $(X,e)\in \R^{5} $ follows a standard normal distribution. Model (\ref{eqmodels1}) has already been considered in Section \ref{s31} in order to highlight the influence of the randomness of the slices on the asymptotic distribution of the estimators. Model (\ref{eqmodels2}) is borrowed from \cite{li1991} and Model (\ref{eqmodels3}) represents a regression model with non-additive noise. Variations of $\sigma$ permits to switch from easy to more difficult situations. We have ran $1000$ Monte-Carlo replication, for which we have performed the test under $H_0$ (when $\eta= (0,0,0,1)$) and under $H_1$ (when $\eta=(1,0,0,0)$) with SIRb1\&2 and CUMEb1\&2 at the nominal level of $\alpha = 5\%$. In each case, the bootstrap sample number was equal to $B=500$. The number of rejections in each situation is given along the tables \ref{tab1} to \ref{tab3}.

\newcolumntype{A}{ >{\columncolor[gray]{.8}}c }

\begin{table}\centering \small
\begin{tabular}{  l c c | c c c c c || c c c c c}
 \hline
\multirow{2}{*}{$\sigma$} &   \multirow{2}{*}{$n$} & \multirow{2}{*}{$H_0$} &  \multicolumn{4}{ >{\columncolor[gray]{.7}}c|  }{SIRb1} & \multicolumn{1}{ >{\columncolor[gray]{.7}}c || }{CUMEb1}& \multicolumn{4}{ >{\columncolor[gray]{.7}}c| }{SIRb2} & \multicolumn{1}{ >{\columncolor[gray]{.7}}c }{CUMEb2} \\ 

 & & & \multicolumn{1}{ >{\columncolor[gray]{.7}}c }{($H$) 3} & \multicolumn{1}{ >{\columncolor[gray]{.7}}c }{5} & \multicolumn{1}{ >{\columncolor[gray]{.7}}c }{7}& \multicolumn{1}{ >{\columncolor[gray]{.7}}c | }{10}& \multicolumn{1}{ >{\columncolor[gray]{.7}}c||  }{}& \multicolumn{1}{ >{\columncolor[gray]{.7}}c }{($H$) 3}&\multicolumn{1}{ >{\columncolor[gray]{.7}}c }{5}& \multicolumn{1}{ >{\columncolor[gray]{.7}}c }{7}&\multicolumn{1}{ >{\columncolor[gray]{.7}}c| }{10} & \multicolumn{1}{ >{\columncolor[gray]{.7}}c }{}\\
 
\hline
\hline
\multirow{8}{*}{$.5$} & \multirow{2}{*}{$30$}& $T$     &14  &       3  &       0  &       0   &     20   &    156  &     192  &     207 &      224  &     173  \\ 
      &                           & $F$    &     1000     &  999   &    988   &    804   &   1000  &    1000   &   1000    &  1000   &   1000  &    1000 
\\  \cline{4-13}
& \multirow{2}{*}{$50$}& $T$      &   15  &       2      &   0    &     0  &      27  &     107  &   129 &      123 &      116  &     108   \\ 
        &                         & $F$    &      1000  &    1000   &   1000   &   1000     & 1000   &   1000  &    1000  &    1000   &   1000   &   1000 \\  \cline{4-13}
& \multirow{2}{*}{$100$}& $T$         &   21  &       7  &       3   &      0   &     25     &   71    &   107     &   79    &    85  &      81   \\ 
          &                       & $F$    &      1000  &    1000   &   1000   &   1000     & 1000   &   1000  &    1000  &    1000   &   1000   &   1000 \\  \cline{4-13}
& \multirow{2}{*}{$200$}& $T$     &   39    &    14   &      8   &      0   &     31   &     72   &     69  &      66   &     63  &      57   \\ 
            &                     & $F$    &      1000  &    1000   &   1000   &   1000     & 1000   &   1000  &    1000  &    1000   &   1000   &   1000 \\  \hline
 \hline
\multirow{8}{*}{$1$} & \multirow{2}{*}{$30$}& $T$     &       17  &       0    &     0  &       0   &     38  &     151  &     180  &     199   &    213  &     151 
 \\ 
      &                           & $F$    &          935    &   794  &     591    &   186   &    990   &    978  &     962  &     960  &     921  &     993 

\\  \cline{4-13}
& \multirow{2}{*}{$50$}& $T$      &    31   &      1  &       0   &      0   &     61    &   132    &   139   &    141 &      136 &      128   \\ 
        &                         & $F$    &           1000  &     991   &    971    &   809    &  1000    &  1000   &    999  &    1000  &     994 &     1000 
 \\  \cline{4-13}
& \multirow{2}{*}{$100$}& $T$         &         22   &      8     &    2    &     0   &     54 &       69 &       64   &     90   &     87 &       84 
   \\ 
          &                       & $F$    &      1000  &    1000   &   1000   &   1000     & 1000   &   1000  &    1000  &    1000   &   1000   &   1000 \\  \cline{4-13}
& \multirow{2}{*}{$200$}& $T$     &     34  &      16  &       4    &     1    &    60   &     77    &    82  &      67&        76 &       89  \\ 
            &                     & $F$    &      1000  &    1000   &   1000   &   1000     & 1000   &   1000  &    1000  &    1000   &   1000   &   1000 \\  \hline

\end{tabular}\caption{Estimated level and power in Model (\ref{eqmodels1}) for $\alpha=5\%$ with $1000$ replications.}\label{tab1}
\end{table}

\begin{table}\centering \small
\begin{tabular}{  l c c | c c c c c || c c c c c}
 \hline
\multirow{2}{*}{$\sigma$} &   \multirow{2}{*}{$n$} & \multirow{2}{*}{$H_0$} &  \multicolumn{4}{ >{\columncolor[gray]{.7}}c|  }{SIRb1} & \multicolumn{1}{ >{\columncolor[gray]{.7}}c || }{CUMEb1}& \multicolumn{4}{ >{\columncolor[gray]{.7}}c| }{SIRb2} & \multicolumn{1}{ >{\columncolor[gray]{.7}}c }{CUMEb2} \\ 

 & & & \multicolumn{1}{ >{\columncolor[gray]{.7}}c }{($H$) 3} & \multicolumn{1}{ >{\columncolor[gray]{.7}}c }{5} & \multicolumn{1}{ >{\columncolor[gray]{.7}}c }{7}& \multicolumn{1}{ >{\columncolor[gray]{.7}}c | }{10}& \multicolumn{1}{ >{\columncolor[gray]{.7}}c||  }{}& \multicolumn{1}{ >{\columncolor[gray]{.7}}c }{($H$) 3}&\multicolumn{1}{ >{\columncolor[gray]{.7}}c }{5}& \multicolumn{1}{ >{\columncolor[gray]{.7}}c }{7}&\multicolumn{1}{ >{\columncolor[gray]{.7}}c| }{10} & \multicolumn{1}{ >{\columncolor[gray]{.7}}c }{}\\
 
\hline
\hline
\multirow{8}{*}{$.1$} & \multirow{2}{*}{$30$}& $T$     &       11    &     2 &        0 &        0  &      36    &   175    &  191 &      209 &      213  &     174 
 \\ 
      &                           & $F$    &            993  &     946    &   870     &  450      &1000     & 1000      & 997    &   995   &    995     & 1000 

\\  \cline{4-13}
& \multirow{2}{*}{$50$}& $T$      &        17   &      1      &   0       &  0      &  33      & 125    &   111    &    99   &    131    &   121   \\ 
        &                         & $F$    &           1000   &   1000  &    1000   &    983  &    1000   &   1000   &   1000  &    1000 &     1000&      1000 
 \\  \cline{4-13}
& \multirow{2}{*}{$100$}& $T$         &        26  &       4 &        0   &      0    &    32    &    90    &    75    &   100    &    81  &      86 
  \\ 
          &                       & $F$    &      1000  &    1000   &   1000   &   1000     & 1000   &   1000  &    1000  &    1000   &   1000   &   1000 \\  \cline{4-13}
& \multirow{2}{*}{$200$}& $T$     &         30  &      21   &      3   &      0    &    39     &   63     &   67    &    71 &       69&        67           \\ 
            &                     & $F$    &      1000  &    1000   &   1000   &   1000     & 1000   &   1000  &    1000  &    1000   &   1000   &   1000 \\  
            \hline
 \hline
\multirow{8}{*}{$.5$} & \multirow{2}{*}{$30$}& $T$     &              16    &     1   &      0    &     0    &    76    &   156     &  162     &  188    &   197  &     161 
 \\ 
      &                           & $F$    &          553   &    296   &    137    &    17    &   786     &  736   &    732  &     723  &     661&       835 

\\  \cline{4-13}
& \multirow{2}{*}{$50$}& $T$      &  27   &      2   &      0   &      0   &     63  &     115   &    114   &    133   &    107    &   103    \\ 
        &                         & $F$    &            803   &    672  &     498    &   176      & 946      & 897       &892     &  872 &      832 &      956 
 \\  \cline{4-13}
& \multirow{2}{*}{$100$}& $T$         &        30   &     11  &       2   &      1    &    72   &     79  &      81  &     109  &      83   &     96 
   \\ 
          &                       & $F$    &            989    &   977  &     954 &      821  &    1000   &    998    &   998   &    993&       991  &    1000 
 \\  \cline{4-13}
& \multirow{2}{*}{$200$}& $T$     &          32   &     17    &     3    &     0    &    48   &     65  &      59  &      63&        72  &      62 
   \\ 
            &                     & $F$    &      1000  &    1000   &   1000   &   1000     & 1000   &   1000  &    1000  &    1000   &   1000   &   1000 \\  \hline

\end{tabular}\caption{Esimtated level and power in Model (\ref{eqmodels2}) for $\alpha=5\%$ with $1000$ replications.}\label{tab2}
\end{table}

\begin{table}\centering \small
\begin{tabular}{  l c c | c c c c c || c c c c c}
 \hline
\multirow{2}{*}{$\sigma$} &   \multirow{2}{*}{$n$} & \multirow{2}{*}{$H_0$} &  \multicolumn{4}{ >{\columncolor[gray]{.7}}c|  }{SIRb1} & \multicolumn{1}{ >{\columncolor[gray]{.7}}c || }{CUMEb1}& \multicolumn{4}{ >{\columncolor[gray]{.7}}c| }{SIRb2} & \multicolumn{1}{ >{\columncolor[gray]{.7}}c }{CUMEb2} \\ 

 & & & \multicolumn{1}{ >{\columncolor[gray]{.7}}c }{($H$) 3} & \multicolumn{1}{ >{\columncolor[gray]{.7}}c }{5} & \multicolumn{1}{ >{\columncolor[gray]{.7}}c }{7}& \multicolumn{1}{ >{\columncolor[gray]{.7}}c | }{10}& \multicolumn{1}{ >{\columncolor[gray]{.7}}c||  }{}& \multicolumn{1}{ >{\columncolor[gray]{.7}}c }{($H$) 3}&\multicolumn{1}{ >{\columncolor[gray]{.7}}c }{5}& \multicolumn{1}{ >{\columncolor[gray]{.7}}c }{7}&\multicolumn{1}{ >{\columncolor[gray]{.7}}c| }{10} & \multicolumn{1}{ >{\columncolor[gray]{.7}}c }{}\\
 
\hline
\hline
\multirow{8}{*}{$.5$} & \multirow{2}{*}{$30$}& $T$     &             23    &     3  &       0 &        0    &    80    &   148   &    190  &     208  &     214    &   134 
 \\ 
      &                           & $F$    &              511      & 410     &  241    &    25  &     264   &    847   &    940    &   946     &  900   &    780

\\  \cline{4-13}
& \multirow{2}{*}{$50$}& $T$      &               23    &     0 &        0  &       0   &     72    &   115     &  130    &   134  &     124  &     107 
  \\ 
        &                         & $F$    &                895  &     924  &     857 &      553 &      748  &     974 &     993 &      993  &     993   &    966 
 
 \\  \cline{4-13}
& \multirow{2}{*}{$100$}& $T$       &               35  &       5 &        1  &       1   &     66  &      79  &      80 &       89  &      79   &     90 
  \\ 
          &                       & $F$    &                    997   &   1000 &     1000    &  1000   &   1000  &     999&      1000 &     1000 &     1000&     1000 
 \\  \cline{4-13}
& \multirow{2}{*}{$200$}& $T$     &            35 &       21   &      5   &      1    &    60   &     67  &      63 &       66 &       67 &       70         \\ 
            &                     & $F$    &      1000  &    1000   &   1000   &   1000     & 1000   &   1000  &    1000  &    1000   &   1000   &   1000 \\

 \hline
 \hline
\multirow{8}{*}{$1$} & \multirow{2}{*}{$30$}& $T$     &                  27    &     3   &      0   &      0   &     93   &    159   &     179   &    190   &    236    &   158 

 \\ 
      &                           & $F$    &                518&       441&       277 &       29  &     289  &     834 &      926  &     939  &     912  &     794 

\\  \cline{4-13}
& \multirow{2}{*}{$50$}& $T$      &         23   &      3    &     2    &     0    &    71    &   100   &     99   &    111  &     119 &       96 
   \\ 
        &                         & $F$    &                 897  &     935  &     865 &      547 &      739 &      974 &      997   &    996 &      996&       974 
 \\  \cline{4-13}
& \multirow{2}{*}{$100$}& $T$         &               27  &       8   &      0     &    0     &   65      &  85    &    84 &      101 &       94  &      81 
   \\ 
          &                       & $F$    &                 999  &    1000   &   1000    &  1000  &    1000  &    1000 &     1000   &   1000   &   1000&      1000 
 \\  \cline{4-13}
& \multirow{2}{*}{$200$}& $T$     &               37   &     14    &     3     &    0     &   52    &    75  &      61 &       69   &     72   &     61 
   \\ 
            &                     & $F$    &      1000  &    1000   &   1000   &   1000     & 1000   &   1000  &    1000  &    1000   &   1000   &   1000 \\  \hline
\end{tabular}\caption{Esimtated level and power in Model (\ref{eqmodels3}) for $\alpha=5\%$ with $1000$ replications.}\label{tab3}
\end{table}

The conclusions might be drawn as follows, first comparing SIR and CUME and second evaluating the differences between b1 and b2.

Comparing SIR and CUME, we must raise from the start, that contrary to SIR, CUME no longer depends on the number of slices $H$. Unfortunately, we see that SIR is strongly affected by changes in $H$, notably at small sample sizes. For instance, under the null in Model (\ref{eqmodels3}), as soon as $H$ is large, b1 no longer rejects the null while b2 reject the null $20\%$ of the times. Looking at the complete picture offered by all the tables, the smaller $H$ the better, making the SIR $3$-slices approach the best competitor for facing CUME. In spite of this ``a posteriori" and advantageous selection of $H$, SIR does not perform better than CUME. For additive models (\ref{eqmodels1}) and (\ref{eqmodels2}), it seems preferable to use CUME over SIR whereas for Model (\ref{eqmodels3}), the situation is slightly mitigated by the high power provided by SIR. To conclude, CUME offers a more simple (no selection of $H$) approach than SIR and among the considered models, it is more accurate to test with CUME rather than SIR.

Both bootstrap tests b1 and b2 converge to the nominal level but with (in average) opposite signs, i.e. b1 tends to underestimate the level while b2's estimated level is always greater than the nominal one. This suggests that b1 is more conservative than b2. Meanwhile, the power associated to b2 is always greater than the power associated to b1. Then and in particular for SIR, both are difficult to compare, and the choice between b1 and b2 should be done with care by the user given the trade-off between conservativeness and powerfulness. For CUME, the situation is rather different than for SIR and the clear winner is b1, notably because of the too high level of type I error committed by b2.

\section{Conclusion}\label{s5}

We have provided a new approach for inverse regression based on empirical processes. This approach has offered a precise description of the asymptotic behaviour of the estimators as well as the validity of the bootstrap. The framework we develop in the paper is linked with the class of indicator functions. This choice was convenient since the metric entropy properties of this class are widely known, but also because of the natural link it induced with the popular methods SIR and CUME. However, the approach developed in this paper can be extended to different classes of functions than indicators. Indeed, for the order $1$ moments based method, one can consider the vector
\begin{align*}
E[X\psi(Y)],
\end{align*}
when $\psi$ varies among a certain family of functions. Another subject of interest for further studies is right-censored data. Suppose we observe
\begin{align*}
\min(Y,C)\quad \t{and} \quad \mathds 1 _{\{Y\leq C\}}, \qquad \t{where}\qquad Y\indep C|X,
\end{align*}
variations of SIR have been studied for instance in \cite{li1999} and \cite{kosorok2011}. It requires a smoothing procedure in order to take into account the effect of the censure.

\paragraph{Acknowledgement.} The author would like to thank Bernard Delyon for helpful comments and advices on this article. He also thank Zhenghui Feng for sharing the Matlab code of CUME.

\input{CVM.bbl}

\bibliographystyle{plain}

\end{document}

%% file: CVM.bbl
\def\cprime{$'$}